\documentclass[11pt,reqno]{amsart}

\usepackage[all]{xy}
\usepackage{hhline}
\usepackage{amssymb}
\usepackage{amsmath,amsthm,amscd,amsfonts}
\usepackage{amsthm}
\usepackage[dvips, dvipsnames, usenames]{color}

\usepackage[active]{srcltx}

\def\kk{\Bbbk}

\newcommand{\GL}{\mathbf{GL}}

\newcommand{\ot}{\otimes}
\newcommand{\Ind}{\operatorname{Ind}}

\newcommand{\ord}{\operatorname{ord}}

\newcommand{\com}{\Delta}

\newcommand\toba{{\mathfrak B }}

\newcommand{\gr}{\operatorname{gr}}

\newcommand{\Lc}{{\mathcal L}}

\newcommand{\eps}{\varepsilon}

\newcommand{\K}{{\mathcal K}}
\newcommand{\Z}{{\mathbb Z}}
\newcommand{\N}{{\mathbb N}}

\newcommand{\M}{{\mathcal M}}

\def\ot{\otimes}
\def\mB{\mathcal{B}}
\def\mA{\mathcal{A}}
\def\mJ{\mathcal{J}}

\def\mC{\mathcal{C}}

\def\mP{\mathcal{P}}

\newcommand{\m}{\mathcal{M}}

\newcommand{\II}{\mathcal{I}}

\newcommand{\T}{{\mathcal T}}

\newcommand{\Oc}{{\mathcal O}}
\newcommand{\oc}{{\mathcal O}}
\newcommand{\ydh}{{}^H_H\mathcal{YD}}

\newcommand{\Alg}{\operatorname{Alg}}

\newcommand\Hom{\operatorname{Hom}}

\newcommand{\Ga}{\Gamma}

\newcommand\co{\operatorname{co}}

\newcommand{\ydhs}{{}^{H^{*}}_{H^{*}}\mathcal{YD}}
\newcommand{\ydao}{{}^{A_0}_{A_0}\mathcal{YD}}
\newcommand{\ydga}{{}^{\Ga}_{\Ga}\mathcal{YD}}
\newcommand{\yddm}{{}^{\dm}_{\dm}\mathcal{YD}}
\newcommand{\ydkm}{{}^{\Bbbk^{\dm}}_{\Bbbk^{\dm}}\mathcal{YD}}
\newcommand{\ydkga}{{}^{\Bbbk^{\Ga}}_{\Bbbk^{\Ga}}\mathcal{YD}}

\def\Zbt2{\tilde{\mathrm{Z}}_b^2}
\def\Hbt2{\tilde{\mathrm{H}}_b^2}
\def\Bbt2{\tilde{\mathrm{B}}_b^2}
\def\Zbh2{\widehat{\mathrm{Z}}_b^2}
\def\Hbh2{\widehat{\mathrm{H}}_b^2}
\def\Bbh2{\widehat{\mathrm{B}}_b^2}

\def\pac{\partial^c}
\def\pa{\partial}
\def\pah{\partial^h}
\def\pab{\partial^b}

\def\ep{\varepsilon}
\def\De{\Delta}
\def\m{\mathrm{m}}
\def\ot{\otimes}

\def\Zep2{\mathrm{Z}^2_\ep}
\def\Hep2{\mathrm{H}^2_\ep}

\def\yd{\mathcal{YD}}
\def\ydh{{_H ^H\hskip -2pt \yd}}

\def\T{\operatorname{T}}

\def\M{\operatorname{M}}

\def\id{\operatorname{id}}
\def\Hom{\operatorname{Hom}}

\theoremstyle{plain}

\newtheorem{lema}{Lemma}[section]
\newtheorem{theorem}[lema]{Theorem}
\newtheorem{cor}[lema]{Corollary}

\theoremstyle{definition}
\newtheorem{definition}[lema]{Definition}
\newtheorem{defn}[lema]{Definition}
\newtheorem{exa}[lema]{Example}

\newtheorem{notation}[lema]{Notation}

\theoremstyle{remark}
\newtheorem{obs}[lema]{Remark}

\theoremstyle{plain}
\newcounter{maint}

\newtheorem{mainthm}[maint]{Theorem}

\theoremstyle{plain}

\def\ep{\varepsilon}

\newcommand\dm{\mathbb D_m}
\newcommand\kdm{\kk^{\dm}}

\newcommand\A{\mathbb A}

\newcommand\s{\mathbb S}
\newcommand{\cS}{\mathcal{S}}

\def\pf{\begin{proof}}
\def\epf{\end{proof}}

\theoremstyle{remark}

\begin{document}

\renewcommand{\baselinestretch}{1.2}

\thispagestyle{empty}

\title[Copointed Hopf algebras over dihedral groups]
{On finite-dimensional copointed Hopf algebras \\ over dihedral groups}

\author[F. Fantino, G. A. Garc\'ia, M. Mastnak]
{Fernando Fantino, Gast\'on Andr\'es Garc\'ia and Mitja Mastnak}

\thanks{This work was partially supported by ANPCyT-Foncyt, CONICET, Secyt (UNLP, UNC), NSERC (Canada)}

\address{\newline\noindent Departamento de Matem\'atica, Facultad de Ciencias Exactas,
Universidad Nacional de La Plata. CONICET. C.C. 172, (1900) La Plata, Argentina.}
\email{ggarcia@mate.unlp.edu.ar}

\address{\newline\noindent Facultad de Matem\'atica, Astronom\'\i a, F\'\i sica y Computaci\'on,
Universidad Nacional de C\'ordoba. CIEM -- CONI\-CET. Medina Allende s/n, Ciudad Universitaria 5000 C\'ordoba, Argentina}
\email{fantino@famaf.unc.edu.ar}

\address{\newline\noindent Department of Mathematics and C.S.
Saint Mary's University Halifax, NS B3H 3C3, Canada.}
\email{mmastnak@cs.smu.ca}

\date{\today}

\begin{abstract}
We classify all finite-dimensional Hopf algebras over an algebraically
closed field of characteristic zero such that its coradical is isomorphic to the algebra
of functions $\kdm$ over a dihedral group $\dm$, with $m=4a\geq 12$.
We obtain this classification by means of the lifting method, where we use cohomology theory
to determine all possible deformations.
Our result provides an infinite family of new examples
of finite-dimensional copointed Hopf algebras over dihedral groups.
\end{abstract}

\maketitle

\section*{Introduction}
Let $\Bbbk$ be an algebraically closed field of characteristic zero.
This paper contributes to the classification of finite-dimensional Hopf algebras over $\Bbbk$.
We address particularly the case where the coradical of the Hopf algebra is the dual of the group algebra
of a non-abelian group $G$.
This kind of algebras are called \emph{copointed} Hopf algebras over $G$.

Let $A$ be a copointed Hopf algebra over a finite group $G$. We say that $A$ is \emph{trivial}
if $A\simeq \Bbbk^{G}$ as Hopf algebras.
It is known that all copointed Hopf algebras over $G$ are trivial if $G$ is an alternating group $\A_n$,
$n\geq 5$, \cite{AFGV-alt},
a sporadic group different from $Fi_{22}$, $B$, $M$, \cite{AFGV-spor}, or a finite projective special
linear group $PSL_n(q)$ for a certain infinite family of pairs $(n,q)$ \cite{ACG}.
On the other hand, there are few non-trivial examples and the complete classification with
non-trivial examples is known only
for the symmetric groups $\s_3$ and $\s_{4}$, see \cite{AV} and \cite{GIV2} respectively, and non-abelian groups
attached to affine racks \cite{GIV}.
In this paper we give the complete classification of copointed Hopf algebras over the
dihedral groups $\dm$, with $m=4a\geq 12$. Throughout the paper $\dm$ will denote the dihedral group of order $2m$.

The case of the dihedral groups was previously treated in \cite{CDMM}, where a particular case is studied.
In this case, the associated graded algebra $\gr A$ is isomorphic to a bosonization of a quantum
line over $\Bbbk^{\mathbb{D}_n}$
with $n\geq 2$. The authors also provide the description in the case where the coradical is the function algebra
over a dicyclic group and when the coradical belongs
 to a certain family of cosemisimple non-pointed self-dual Hopf algebras.

The question of classifying Hopf algebras over $\Bbbk$ is a difficult problem to attack.
One of the main reasons lies on the lack of general methods.
The best approach to study Hopf algebras with the Chevalley property,
\emph{i.~e.}, the coradical is a Hopf subalgebra,
is the lifting method developed by Andruskiewitsch and Schneider \cite{AS}.
Any Hopf algebra $A$ has a coalgebra filtration $\{A_{ n } \}_{n\geq 0} $,
called \textit{the coradical filtration}, whose first term is the
\emph{coradical} $A_0$. It
corresponds to the filtration of $A^{*}$ given by the powers of the Jacobson radical. If $A_{ 0}$ is a
Hopf subalgebra, then $\{A_{ n } \}_{n\geq 0} $ is also an algebra filtration; in particular,
its associated graded object $\gr A = \bigoplus_{ n\geq 0} A_{ n} /A_{ n-1}$ is a graded Hopf
algebra, where  $A_{-1} = 0$. Let $\pi : \gr A \to A_{ 0}$ be the homogeneous projection.
It turns out that $\gr A \simeq B\# A_{0}$ as Hopf algebras, where
$B = (\gr A)^{\co \pi} = \{a\in A:\ (\id\ot \pi)\com(a)= a\ot 1\}$ and
$\#$ stands for the Radford-Majid biproduct or \textit{bosonization} of $B$ with $A_{0}$.
Here $B$ is not a usual Hopf algebra, but a graded connected Hopf algebra in
the category $\ydao$ of Yetter-Drinfeld modules over $A_{0}$. It contains the
algebra generated by the elements of degree one, called the \textit{Nichols algebra} $\toba(V)$  of $V$; here
$V=B^{1}$ is a braided vector space whose braiding $c:V\ot V \to V\ot V$ is called the \textit{infinitesimal braiding} of $A$.

Let $H$ be a finite-dimensional cosemisimple Hopf algebra. The main steps to
determine all finite-dimensional Hopf algebras $A$ with coradical $A_0\simeq H$, in terms of the lifting method, are:
\begin{enumerate}
\item[$(a)$] determine all $V\in\ydh$ such that the
Nichols algebra $\toba(V)$ is finite-dimensional.
\item[$(b)$] for such $V$, compute all Hopf algebras $A$ such
that $\gr A\simeq \toba(V)\# H$. We call $A$ a \emph{lifting}
of $\toba(V)$ over $H$.
\item[$(c)$] prove that if $A$ is any finite-dimensional Hopf algebra such that $A_0\simeq H$ and
$\gr A \simeq B\#H$ with $B$ a braided Hopf algebra in $\ydh$, then $B\simeq \toba(V)$ for some
$V\in \ydh$.
\end{enumerate}
By \cite[Proposition 2.2.1]{AG}, we have that $\ydh\simeq \ydhs$ as braided tensor categories.
Hence, the questions posed in previous steps $(a)$ and $(c)$ can be answered either in
the category $\ydh$ or in the category $\ydhs$.
Since all pointed Hopf algebras over $\dm$ were classified in \cite{FG}
by means of the lifting method, to study copointed Hopf algebras over $\dm$
we have to deal only with step (b) and (c). To describe all possible liftings,
we use the invariant Hochschild cohomology of the Nichols algebra. As a consequence,
it turns out that all liftings are cocycle deformations.

Recently, it was proved in \cite{AGI} that all liftings of Nichols algebras of diagonal type over a
cosemisimple Hopf algebra $H$ can be obtained as a $2$-cocycle deformation of the bosonization
$\toba(V)\# H$, provided $V$ is a
principal realization in $\ydh$ and the Nichols algebra $\toba(V )$ is
finitely presented. In particular, all liftings may be described by generators and relations
using the strategy developed in \cite{AAGMV}. Although in our case
the braiding is diagonal and $\kdm$ is cosemisimple,
the infinitesimal braidings can not be described as principal realizations in $\ydkm$, since
$\kdm$ does not coact homogeneously. Hence, the copointed Hopf algebras we introduce might not be obtained
directly from their method. The strategy developed in \cite{GIV2} seems to be more appropriate.


A \textit{formal graded deformation} of a graded bialgebra $A$ over $\Bbbk[t]$,
consists of a $\Bbbk[t]$-bilinear multiplication
$m_{t} = m + tm_{1} + t^{2}m_{2} + \cdots$
and a comultiplication $\com_{t} = \com + t\com_{1} + t^{2}\com_{2}+ \cdots$
with respect to which the $\Bbbk[t]$-module $A[t] := A \ot_{\Bbbk} \Bbbk[t]$ is again a
graded bialgebra (where the degree of $t$ is $1$).   At first glance it might seem that the coefficients are
in the formal power series $\Bbbk[[t]]$ (as is the case in the ``ungraded" version of the theory), but the fact that structure maps $m_t$, $\com_t$,
are homogeneous of degree $0$, yields that for every $i$, the maps $m_i, \com_i$ are homogeneous of degree $-i$, and hence
the images of $A$ and $A\otimes A$ under $m_t$ and $\com_t$ always lie in $A[t]$ and $(A\otimes A)[t]$.
For the sake of simplicity let us now assume that the deformation is a preferred deformation in the sense that $\com_t=\com$
(formal graded deformations of graded bialgebras satisfying the assumptions of Theorem A are equivalent to such deformations).
Given such a deformation of $A$, let $r$ be the smallest
positive integer for which $m_{r}\neq 0$ (if such an $r$ exists).
Then $(m_{r}, 0)$ is a $2$-cocycle in the bialgebra cohomology
$\Zbh2(A)$ called an \emph{infinitesimal deformation}. Every nontrivial deformation
is equivalent to one for which the corresponding
$(m_{r}, 0)$ represents a nontrivial cohomology class, see \cite{DCY} (see also \cite{GS} for the ungraded version).
Conversely, given a
positive integer $r$ and a $2$-cocycle
$(m', 0)$ in $\Zbh2(A)$, $m + t^{r}m'$
is an associative multiplication on the \textit{$r$-deformation} $A[t]/(t^{r+1})$ of $A$, making it
a bialgebra over $\Bbbk[t]/(t^{r+1})$.
There may or may not exist $m_{r+1}, m_{r+2}, \ldots$ for which
$m + t^{r}m' + t^{r+1}m_{r+1} + t^{r+2}m_{r+2} + \cdots$
endows $A[t]$ with a bialgebra structure over $\Bbbk[t]$.
In principle it is hard to know if an $r$-deformation corresponds to a formal deformation.

Clearly, any formal deformation $U$ of a graded bialgebra $A$ is a lifting in the sense that $\gr U \simeq A$. Conversely,
it is well-known that any lifting corresponds to a formal deformation, see for example
\cite[Theorem 2.2]{DCY}. If $\sigma:A\ot A \to \Bbbk$ is a multiplicative $2$-cocycle on $A$, see
Subsection \ref{subsec:def-cocycles}, then the bialgebra $A_{\sigma}$ is a lifting of $A$ and
whence corresponds to a formal deformation whose infinitesimal part $\sigma_{r}$ is a
Hochschild $2$-cocycle, see Section \ref{sec:cohomology}.

Assume $A=B\# H$ is a bosonization of a quadratic Nichols algebra $B$ over $H$. If
the braiding in $\ydh$ is symmetric, then any $r$-deformation given by a $2$-cocycle
$(m', 0)$ in $\Zbh2(A)$ corresponds to a formal deformation, since $m'$ is
the infinitesimal part of the multiplicative $2$-cocycle
given by the exponentiation $\sigma=e^{m'}$ see \cite[Corollary 2.6]{GM}.
Hence, to find all possible deformations it suffices to describe
the set $ \Hbh2(A)$.
For more details on bialgebra deformations see \cite{GS}, \cite{DCY}, \cite{MW}.

We will restrict our attention to the $H$-bitrivial part of cohomology, \emph{i.~e.},
pairs $(f,g)$ where $f|_{A_0\ot A}=0=f|_{A\ot A_0}$ and $(\pi\ot\id)g=0=(\id\ot\pi)g$,
where $\pi\colon A\to A_0$ denotes the canonical projection.
We denote the corresponding sets by $\Zbt2(A)$ and $\Bbt2(A)$,
and the relevant cohomology by
$\Hbt2(A)=\Zbt2(A)/\Bbt2(A)$.
If $H= A_{0}$ is a group algebra or a  dual of a group algebra, then by \cite[Lemma 2.3.1]{MW},
it follows that $\Hbt2(A)=\Hbh2(A)$.
For $\ell\in\mathbb{Z}$, we denote the relevant graded cohomology by
$
\Hbt2(A)_\ell=\Zbt2(A)_\ell/\Bbt2(A)_\ell$.
Let $\Hep2(B,\kk)$ denote the Hochschild 2-cohomology group of $B$ and $\Hep2(B,\kk)^H$ the $H$-stable part.
Under our hypothesis, it holds that there exists a connecting homomorphism
$\tilde{\pa}_\ell\colon \Hep2(B,\kk)^H_\ell\to \Hbt2(A)_\ell$.
Suppose $B=\T(V)/I$
for some space $V\in\ydh$ and
an ideal $I\subseteq \T(V)_{(2)}=\bigoplus_{n\ge 2} V^{\otimes n}$.  Set
$\M(B)= I/\big(T(V)^+ I + I T(V)^+\big)$ and
denote by $\yd(\M(B),V)_\ell$ the $\ell$-homogeneous morphisms between $M(B)$ and $V$ in $\ydh$.
As a particular case of our first result we obtain the following theorem that generalizes
\cite[Theorem 6.2.1]{MW}.

\begin{mainthm}\label{thm:M(B)V}
Let $\ell<0$ and $r=-\ell$.  If $\yd(\M(B),\mP(B))_\ell=0$ and $B$ is generated as an algebra by elements of degree at most $r$, then
$\tilde{\pa}_\ell\colon \Hep2(B,\kk)^H_\ell\to \Hbt2(A)_\ell$ is surjective.
\end{mainthm}

This implies that all homogeneous
$2$-cocycles $(m',0) \in  \Hbt2(A)_2 $, and consequently also all deformations,
can be described using $H$-invariant Hochschild $2$-cocycles
on $B$, see Remarks \ref{rmk:formula-def-gen} and \ref{rmk:formula-def-part}.

Using the machinery described above, we obtain our main result. It turns out that all finite-dimensional Hopf algebras with coradical
$A_{0}$ isomorphic to $\kdm$ are liftings of bosonization of Nichols algebras over semisimple Yetter-Drinfeld modules.
These semisimple Yetter-Drinfeld modules are given by $M_{I}=\bigoplus_{1\leq j\leq r}M_{(i_{j},k_{j})}$,
$M_{L} = \bigoplus_{1\leq i \leq r} M_{\ell_{i}}$ and $M_{I,L} = M_{I} \oplus M_{L}$, and are parametrized by sets $I$, $L$ of
tuples satisfying certain properties, see Subsection \ref{subsec:ydmodkdm}. All Nichols algebras over these modules are isomorphic to exterior algebras.
The liftings are then encoded by families of parameters $\zeta_{I}$, $\mu_{L}$, $\nu_{L}$, $\tau_{L}$ related to this description.

\begin{mainthm}\label{thm:haoverkdm}
Let $A$ be a finite-dimensional Hopf algebra with coradical $A_{0}$ isomorphic to $\kdm$ with $m=4a\geq 12$.
Then $A$ is isomorphic to one of the following Hopf algebras:
\begin{enumerate}
 \item[(i)] $\mA(\zeta_{I})$ for some $I\in \II$ and some lifting data $\zeta_{I}$,
 \item[(ii)] $\mB(\mu_{L},\nu_{L},\tau_{L})$ for some $L\in \Lc$ and
 some lifting data $(\mu_{L},\nu_{L},\tau_{L})$,
 \item[(iii)] $\mC(\zeta_{I},\mu_{L},\nu_{L},\tau_{L})$ for some
 $(I,L)\in \K$ and some lifting data $(\zeta_{I},\mu_{L},\nu_{L},\tau_{L})$.
 \end{enumerate}
 Conversely, any Hopf algebra above is a lifting of a finite-dimensional Nichols algebra in
 $\ydkm$.
\end{mainthm}

The paper is organized as follows. In Section \ref{sec:prelim} we set the conventions and give the
necessary definitions and results used along the paper.
Section \ref{sec:cohomology} is devoted to develop the cohomological tools to describe all
possible deformations.
In Section \ref{sec:pointed-dihedral} we begin by recalling the description of the simple objects in $\ydkm$
and the Nichols algebras associated to them.
Then we introduce the deformed algebras listed in Theorem \ref{thm:haoverkdm}
and we prove they are all the possible finite-dimensional copointed Hopf algebras over $\dm$.

\section{Preliminaries}\label{sec:prelim}

\subsection{Conventions}
We work over an algebraically closed field $\kk$ of characteristic
zero. Our references for Hopf algebra theory are
\cite{Mo}, \cite{R} and \cite{S}.

Let $A$ be a Hopf algebra over $\kk$. We denote by $\com$, $\eps$
and $\cS$ the comultiplication,
the counit and the antipode, respectively.
Comultiplication and coactions are written using the
Sweedler-Heynemann notation with summation sign suppressed, \textit{e.g.},
$\com(a)=a_{(1)}\ot a_{(2)}$ for $a\in A$.

The \emph{coradical}
$A_0$ of $A$ is the sum of all simple
sub-coalgebras of $A$. In particular, if $G(A)$ denotes the group
of \emph{group-like elements} of $A$, we have $\kk G(A)\subseteq
A_0$. A Hopf algebra is \emph{pointed} if $A_0=\kk
G(A)$, that is, all simple sub-coalgebras are one-dimensional. We say that
$A$ has the \textit{Chevalley property} if $A_{0}$ is a Hopf subalgebra of $A$. In particular,
we say that $A$ is \textit{copointed} if $A_{0}$ is isomorphic as Hopf algebra to the algebra of functions $\kk^{\Ga}$ over some
finite group $\Ga$.

Denote by $\{A_i\}_{i\geq 0}$ the \emph{coradical
filtration} of $A$; if $A$ has the Chevalley property, then by $\gr A = \oplus_{n\ge 0}\gr A(n)$ we denote the associated
graded Hopf algebra, with
 $\gr A(n) = A_n/A_{n-1}$, setting
$A_{-1} =0$.
For $h,g\in G(A)$, the linear space of $(h,g)${\it -primitives}
 is:
$$
\mathcal{P}_{h,g}(A) :=\{x\in A\mid\com(x)= x\ot h + g\ot x\}.
$$
If $g=1=h$, the linear space $\mathcal{P}(A) = \mathcal{P}_{1,1}(A)$ is called
the set of primitive elements.

Let $D$ be a coalgebra over $\Bbbk$ and denote by $(D_{n})_{n\in \N}$
its coradical filtration.
Then there exists a coalgebra projection
$\pi:D \to D_{0}$ from
$D$ to the coradical
$D_{0}$ with kernel $I$, see
\cite[Theorem 5.4.2]{Mo}.
Define the maps
$$
\rho_{L}:= (\pi\ot \id)\com: D \to D_{0}\ot D \qquad\mbox{ and }\qquad
\rho_{R}:= (\id\ot \pi)\com: D \to D\ot D_{0},
$$
\noindent and let $P_{n}$ be the sequence of subspaces defined recursively
by
\begin{align*}
P_{0} & = 0,\\
P_{1} & = \{x\in D:\ \com(x) = \rho_L(x) +\rho_R(x)\}
= \com^{-1}(D_{0}\ot I + I\ot D_{0}),\\
P_{n} & = \{x\in D:\ \com(x) - \rho_L(x) - \rho_R(x) \in
\sum_{1\leq i \leq n-1}P_{i}\ot P_{n-i}\}, \quad n\geq 2.
\end{align*}
By a result of Nichols, $P_{n} = D_{n}\cap I$, $D_n = D_0 \oplus P_n$ for $n\geq 0$,
and $D_{n},P_{n}$ are
$D_{0}$-sub-bicomodules of $D$ via $\rho_R$ and $\rho_L$.

If $M$
is a left $A$-comodule via $\delta(m)=m_{(-1)}\ot m_{(0)} \in A\ot M$
for all $m\in M$,
then the space of {\it left
coinvariants} is $\ ^{\co \delta}M = \{x\in
M\mid\delta(x)=1\ot x\}$. In particular, if $\pi:A\rightarrow H$ is a morphism of Hopf
algebras, then $A$ is a left $H$-comodule via $(\pi\ot\id)\com$ and
in this case $$^{\co \pi}A:=\ ^{\co\ (\id\ot\pi)\com}A
=  \{a\in
A\mid (\pi\ot \id )\com(a)=1\ot a\}.$$
Right coinvariants, written
${A^{\co \pi} }$ are defined analogously.

\subsection{Yetter-Drinfeld modules and
Nichols algebras}\label{subsec:ydCG-nichols-alg}
In this subsection we recall the definition of Yetter-Drinfeld modules
over Hopf algebras $H$ and we describe the irreducible ones
in case $H$ is a group algebra. We also add the definition of
Nichols algebras associated with them.

\subsubsection{Yetter-Drinfeld modules over Hopf algebras}
Let $H$ be a Hopf algebra. A \textit{left Yetter-Drinfeld module} $M$ over $H$
is a left $H$-module $(M,\cdot)$ and a left $H$-comodule
$(M,\lambda)$ with $\lambda(m)=m_{(-1)}\ot m_{(0)} \in H\ot M$
for all $m\in M$, satisfying the compatibility
condition
$$ \lambda(h\cdot m) = h_{(1)}m_{(-1)}\cS(h_{(3)})\ot h_{(2)}\cdot m_{(0)}
\qquad \forall\ m\in M, h\in H. $$
We denote by $\ydh$ the corresponding category.
It is a braided monoidal category: for any $M, N \in \ydh$,
the braiding $c_{M,N}:M\ot N \to N\ot M$ is given by
$$c_{M,N}(m\ot n) = m_{(-1)}\cdot n \ot m_{(0)} \qquad \forall\
m\in M,n\in N.$$
A \emph{braided} Hopf algebra is a Hopf algebra in $\ydh$ such that all structural maps
are morphisms in the category.

Suppose $H$ is a finite-dimensional Hopf algebra.
Then the category $\ydhs$  is braided equivalent to $\ydh$, see \cite[Remark 1.1.5 and Theorem 2.2.1]{AG}.
 The functor $F:\ydh\to\ydhs$ that yields the braided tensor equivalence is given as follows: for $V\in\ydh$ define $F(V)$ to be $V$ as
 $\Bbbk$-vector space and with its Yetter-Drinfeld module structure defined by
\begin{equation}\label{eq:dualyd}
f\cdot v = f(\cS(v_{(-1)})) v_{(0)},\quad
\lambda(v) = \sum_{i} \cS^{-1}(f_{i})\ot h_{i}\cdot v\quad \text{ for all }v\in V,\ f\in H^{*},
\end{equation}
where $(h_{i})_{i\in I}$ and $(f_{i})_{i\in I}$ are dual bases of $H$ and $H^{*}$.

Assume $H=\kk \Ga$ with $\Ga$ a finite group. In this particular case, a left Yetter-Drinfeld module over $\kk \Ga$ is a left
$\kk\Ga$-module and left $\kk \Ga$-comodule $M$ such that
$$
\lambda(g \cdot m) = ghg^{-1} \otimes g \cdot m, \qquad \forall\ m\in M_h, g, h\in \Ga,
$$
where $M_h = \{m\in M: \lambda(m) = h\otimes m\}$; clearly, $M =
\oplus_{h\in \Ga} M_h$ and thus $M$ is a $\Ga$-graded vector space such that $g\cdot M_h \subseteq M_{ghg^{-1}}$, for all $g, h\in \Ga$.
We denote this category simply
by $\ydga$. As $\kk\Ga$ is semisimple,
Yetter-Drinfeld modules over $\kk\Ga$ are
completely reducible. The irreducible modules
are parameterized by pairs $(\oc, \rho)$, where $\oc$ is a
conjugacy class of $\Ga$ and $(\rho,V)$ is an irreducible representation of the
centralizer $C_\Ga(z)$ of a fixed point $z\in \oc$.
The corresponding Yetter-Drinfeld module is given by
$M(\oc, \rho) = \Ind_{C_\Ga(z)}^{G}V =
\Bbbk \Ga \otimes_{C_\Ga(z)} V$.
Explicitly, let $z_1 = z$, \dots, $z_{n}$
be a numeration of $\oc$ and
let $g_i\in \Ga$ such that $g_i z g_{i}^{-1} = z_i$ for all $1\le i \le
n$. Then  $M(\oc, \rho) = \oplus_{1\le i \le n}g_i\otimes V $. The Yetter-Drinfeld module
structure is given as follows. Write
$g_iv := g_i\otimes v \in M(\oc,\rho)$, $1\le i \le n$, $v\in V$.
For $v\in V$ and $1\le i \le n$, the action of $g\in \Ga$ and
the coaction are given by
$$g\cdot (g_iv) = g_j(\gamma\cdot v), \qquad
\lambda(g_iv) = z_i\otimes g_iv,
$$
where $gg_i = g_j\gamma$, for unique $1\le j \le n$ and $\gamma\in
C_\Ga(z)$.

Consider now the function algebra $\Bbbk^{\Ga}$
and denote by $\{\delta_{g}\}_{g\in \Ga}$ the dual basis of the canonical
basis of $\Bbbk \Ga$. By the braided equivalence described
above, we have also a parametrization of the irreducible objects in
$\ydkga$. Explicitly, $M(\oc, \rho)  \in \ydkga$ with the structural maps given by
$$  \delta_{g} \cdot (g_{ i} v ) = \delta_{ g,z_{i}^{-1}} g_{ i} v ,\qquad
\lambda(g_{ i}v ) =
\sum_{g\in \Ga} \delta_{g^{-1}} \ot g \cdot (g_{i}v ).$$

\subsubsection{Nichols algebras}\label{subsubsec:nichols-alg}
The Nichols algebra of a braided vector space $(V,c)$ can be
defined in different ways, see \cite{N, AG, AS, H}.
As we are interested in Nichols algebras of
braided vector spaces arising from Yetter-Drinfeld modules,
we give the explicit definition in this case.

\begin{definition}\cite[Definition 2.1]{AS}
Let $H$ be a Hopf algebra and $V \in \ydh$. A braided $\N$-graded
Hopf algebra $B = \bigoplus_{n\geq 0} B(n) \in \ydh$  is called
the \textit{Nichols algebra} of $V$ if
\begin{enumerate}
 \item[(i)] $\kk\simeq B(0)$, $V\simeq B(1) \in \ydh$,
 \item[(ii)] $B(1) = \mathcal{P}(B)
=\{r\in B~|~\com_{B}(b)=b\ot 1 + 1\ot b\}$, the linear space of primitive elements.
 \item[(iii)] $B$ is generated as an algebra by $B(1)$.
\end{enumerate}
In this case, $B$ is denoted by $\toba(V) = \bigoplus_{n\geq 0} \toba^{n}(V) $.
\end{definition}

For any $V \in \ydh$ there is a Nichols algebra $\toba(V)$ associated with it.
It can be constructed as a quotient of the tensor algebra $T (V)$ by the largest
homogeneous two-sided ideal $I$ satisfying:
\begin{itemize}
\item $I$ is generated by
homogeneous elements of degree $\geq 2$.
\item $\com(I) \subseteq I\ot T(V) + T(V)\ot I$, \emph{i.~e.}, it is also a coideal.
\end{itemize}
In such a case, $\toba(V) = T(V)/ I$. See
\cite[Section 2.1]{AS} for details.

An important observation is that the Nichols algebra
$\toba(V)$, as algebra and coalgebra,
is completely determined just by the braiding.

Let $\Ga$ be a finite group.
We denote by
$\toba(\oc,\rho)$ the
Nichols algebra associated
with the irreducible Yetter-Drinfeld module $M(\oc, \rho) \in \ydga$.

\subsection{Bosonization and Hopf algebras with a projection}\label{subsec:bosonization}
Let $H$ be a Hopf algebra and $B$ a braided Hopf algebra in $\ydh$.
The procedure to obtain a
usual Hopf algebra from $B$ and $H$ is called
the Majid-Radford biproduct or \emph{bosonization}, and it is usually
denoted by $B \#H$. As a vector space $B \# H = B\otimes H$, and the
multiplication and comultiplication are given by the smash-product
and smash-coproduct,
respectively. That is, for all $b, c \in B$ and $g,h \in H$
\begin{align*}
(b \# g)(c \#h) & = b(g_{(1)}\cdot c)\# g_{(2)}h,\\
\com(b \# g) & =b^{(1)} \# (b^{(2)})_{(-1)}g_{(1)} \ot
(b^{(2)})_{(0)}\# g_{(2)},
\end{align*}
where $\com_{B}(b) = b^{(1)}\ot b^{(2)}$ denotes the comultiplication in $B\in \ydh$.
We identify $b=b\# 1$ and
$h=1\# h$; in particular we have $bh=b\# h$ and $hb=h_{(1)}\cdot b\# h_{(2)}$.
Clearly, the map $\iota: H \to B\#H$ given by $\iota(h) = 1\#h$
is an injective Hopf algebra map, and the map $\pi: B\#H \to H$
given by $\pi(b\#h) = \eps_{B}(b)h$
is a surjective Hopf algebra map such that $\pi \circ \iota = \id_{H} $.
Moreover, it holds that $B = (B\#H)^{\co \pi}$.

Conversely, let $A$ be a Hopf algebra with bijective antipode and
$\pi: A\to H$ a Hopf algebra epimorphism admitting
a Hopf algebra section $\iota: H\to A$.
Then $B=A^{\co\pi}$ is a braided Hopf algebra in $\ydh$ and $A\simeq B\# H$
as Hopf algebras.

\subsection{Deforming cocycles}\label{subsec:def-cocycles}
Let $A$ be a Hopf algebra.
Recall that a convolution invertible linear map
$\sigma $ in $\Hom_{\Bbbk}(A\ot A, \Bbbk)$
is a
\textit{normalized multiplicative 2-cocycle} if
$$ \sigma(b_{(1)},c_{(1)})\sigma(a,b_{(2)}c_{(2)}) =
\sigma(a_{(1)},b_{(1)})\sigma(a_{(2)}b_{(2)},c)  $$
and $\sigma (a,1) = \eps(a) = \sigma(1,a)$ for all $a,b,c \in A$,
see \cite[Section 7.1]{Mo}.
In particular, the inverse of $\sigma $ is given by
$\sigma^{-1}(a,b) = \sigma(\cS(a),b)$ for all $a,b\in A$.

Let $B$ be an algebra and consider $\Hom_{\Bbbk}(A^{\ot n}, B)$
as a $\Bbbk$-algebra with the convolution product given by
$(f*g)(a)=f(a_{(1)})g(a_{(2)})$ for all $a\in A^{\ot n}$. With this notation, the 2-cocycle condition above reads
\begin{equation}\label{eq:2-cocycle}
(\eps  \ot \sigma) * [\sigma \circ ( \id_{A} \ot m) ]= (\sigma \ot \eps) * [\sigma \circ (m \ot \id_{A})].
\end{equation}

Using a $2$-cocycle $\sigma$ it is possible to define
a new algebra structure on $A$, which we denote by $ A_{\sigma} $, by deforming the multiplication.
Moreover, $A_{\sigma}$ is indeed
a Hopf algebra with
$A = A_{\sigma}$ as coalgebras,
deformed multiplication
$m_{\sigma}: A \ot A \to A$
given by $$m_{\sigma}(a,b) = a\cdot_{\sigma}b = \sigma(a_{(1)},b_{(1)})a_{(2)}b_{(2)}
\sigma^{-1}(a_{(3)},b_{(3)})\qquad\text{ for all }a,b\in A,$$
and antipode
$\cS_{\sigma}: A \to A$ given by (see \cite[Theorem 1.6 (b)]{doi} for details)
$$\cS_{\sigma}(a)=\sigma(a_{(1)},\cS(a_{(2)}))\cS(a_{(3)})
\sigma^{-1}(\cS(a_{(4)}),a_{(5)})\qquad\text{ for all }a\in A.$$

\subsubsection{Deforming cocycles for graded Hopf algebras}\label{subsubsec:defgraded}
Let $\displaystyle{A=\bigoplus_{n\geq 0} A_n}$ be a
$\mathbb{N}_{0}$-graded Hopf algebra, where
$A_n$ denotes the homogeneous component of $A$ of degree $n$.
Let $\sigma\colon A\otimes A\to \Bbbk$ be a normalized
multiplicative $2$-cocycle and
assume that $\sigma|_{A_0\otimes A_0}=\eps\ot \eps$.

We decompose $\sigma=\sum_{i=0}^\infty \sigma_i$ into
the (locally finite) sum of homogeneous maps $\sigma_i$, with
$$
\sigma_i : (A\otimes A)_i = \bigoplus_{p+q=i} A_p\otimes A_q
\stackrel{\sigma|_{(A\otimes A)_i}}{\to} \Bbbk.
$$
We set $-i$ for the \textit{degree} of $\sigma_{i}$.
Note that due to our assumption $\sigma|_{A_0\otimes A_0}=\eps\ot \eps$
we have $\sigma_0=\eps \ot \eps$.
Decomposing $\sigma^{-1}=\sum_{j=0}^\infty \eta_j$, where $\eta_j=(\sigma^{-1})|_{(A\ot A)_j}$, we have that
$$\displaystyle{\sum_{i+j=\ell} \sigma_i * \eta_j =\delta_{\ell, 0}\; \eps\ot\eps =\sum_{i+j=\ell} \eta_i * \sigma_j },$$ for $\ell \geq 1$.
Note that $\eta_0=\eps\ot\eps$ and that for the least positive integer $s$ for which $\sigma_s\ne 0$
we have $\eta_s=-\sigma_s$.
Moreover, the cocycle condition \eqref{eq:2-cocycle}
implies that
$$\sum_{i+j=\ell}(\eps\ot\sigma_i )*[\sigma_j \circ (\id\ot m)]=
\sum_{i+j=\ell}(\sigma_i\ot\eps )*[\sigma_j \circ (m\ot \id)]$$
for all $\ell\ge 1$. In particular,
$$\eps\ot\sigma_s + [\sigma_s \circ (\id\ot m)]=\sigma_s\ot\eps +[\sigma_s \circ (m\ot \id)]$$
which implies that $\sigma_s\colon A\ot A\to \Bbbk$ is a
\textsl{Hochschild 2-cocycle} of $A$.
We call $\sigma_s$ the
\textsl{graded infinitesimal part} of $\sigma$. We denote the Hochschild
2-cohomology group by $ \Hep2(A,\kk)$.

\section{Computing cohomology}\label{sec:cohomology}
Let $A = \bigoplus_{n\geq 0} A_{n}$ be a graded Hopf algebra. We say that a filtered Hopf algebra $K$ is a \textit{lifting} of $A$ if
$\gr K \simeq A$ as graded Hopf algebras. The aim of this section is to give a recipe for computing the set of
$A_{0}$-bitrivial bialgebra $2$-cocycles $\Hbt2(A)$, see definitions below. For more details on
bialgebra cohomology and the definition of the Hochschild and coalgebra differentials $\partial^{h}$ and
$\partial^{c}$, respectively, we refer to \cite[Section 2]{MW}.

\begin{obs}\label{rmk:lifting=def}
By \cite[Theorem 2.2]{DCY}, there is a natural bijection between isomorphisms classes of
liftings of a graded bialgebra $A$ and isomorphisms classes of formal deformations of $A$.
Indeed, for a lifting $K$ of $A$, we denote by $m_{K}$ the product and we identify $K_\ell$ with $\oplus_{i=1}^\ell A_i$.
 Since $K$ is a filtered bialgebra,
one has that $m_{K}: A_{i}\ot A_{j} \to K_{i+j}$. Thus, there exist unique homogeneous
maps $m_j= \pi_{j} \circ m_K$ for all $j\geq 0$, such that $m_{K}(a\ot b) = \sum_{j\geq 0} m_{j}(a\ot b)$; here
$\pi_{j}: K_{j} \to A_{j}$ denotes the canonical projection with kernel $\oplus_{i=1}^{j-1} A_i$. Since
$\gr K = A$ as graded algebras, one has that $m_{0}=m$.
The corresponding formal deformation for $K$ is given by $(A[t],m_t)$ with
$$m_t(a\ot b)=\sum_{j \geq 0}m_j(a\ot b)t^j, \qquad a,b\in A.$$
\end{obs}

\begin{obs}\label{rmk:relation-cocycle-def}
Let $\sigma\colon A\otimes A\to \Bbbk$ be a normalized multiplicative $2$-cocycle.
The cocycle deformation $A_\sigma$ is a filtered bialgebra with the
underlying filtration inherited from the grading on $A$, \emph{i.~e.}, the
$\ell$-th filtered part is $(A_\sigma)_{(\ell)}=\bigoplus_{i=0}^\ell A_i$.
Note that the associated graded bialgebra $\gr A_\sigma$ can be identified with $A$,
that is, $A_{\sigma}$ is a lifting of $A$.
As in Remark \ref{rmk:lifting=def},
decomposing the multiplication
$m_\sigma=\sigma*m*\sigma^{-1}$ into the sum of homogeneous components
$m_i$ of degree $-i$, allows us to identify the filtered $\Bbbk$-linear structure $A_\sigma$ with
a $\Bbbk[t]$-linear structure $(m_\sigma)_{t} \colon A[t]\otimes A[t] \to A[t]$ induced by
$(m_\sigma)_{t}|_{A\otimes A}=
m+tm_1  +t^2m_2 +\ldots$ (as before, we assume that the degree of $t$ is 1).
If $\Delta_{t}\colon A[t]\to A[t]\otimes A[t]$ is the
$\Bbbk[t]$-linear map induced by $\Delta_{t}|_{A} = \Delta$,
then the graded Hopf algebra $A[t]_\sigma= (A[t], (m_\sigma)_{t}, \Delta_{t})$
is a graded formal deformation of $A$ in the sense of \cite{DCY}, see also \cite{GS, MW}.
Note that in case $m$ does not commute with the
graded infinitesimal part $\sigma_s$ of $\sigma$,
then $( \sigma_s*m-m*\sigma_s,0)$,
is the infinitesimal part of the formal graded deformation
(and is a $2$-cocycle in $\Zbh2(A)_s$).
\end{obs}

\begin{notation}\label{notation-le} Let $U,V$ be graded vector spaces and
$f\colon U\to V$ a linear map. For $r\in\mathbb{N}$,
define $f_r\colon U\to V$ as the linear map given by $f_r|_{U_r}=f$
and $f_r|_{U_s}=0$ for $r\not=s$.  Similarly we define a linear map
$f_{\le r}$ by $f_{\le r}|_{U_s}=f$ for $s\le r$ and $f_{\le r}|_{U_s}=0$ for $s>r$.
The map $f_{< r}$ is defined by $f_{< r}|_{U_s}=f$ for $s< r$ and
$f_{< r}|_{U_s}=0$ for $s\ge r$.
\end{notation}

Recall that the second truncated Gerstenhaber-Schack bialgebra
cohomology $$\Hbh2(A)=\Zbh2(A)/\Bbh2(A)$$ is given by
\begin{align*}
\widehat{{\rm Z}}_b^2(A)= \big\{(f,g)\;\big|\;& f\colon A^+\ot A^+\to A^+,\
g\colon A^+\to A^+\ot A^+,\nonumber\\
&af(b,c)+f(a,bc)=f(ab,c)+f(a,b)c,\\ 
&c_{(1)}\ot g(c_{(2)})+(\id\ot\De)g(c)=(\De\ot\id)g(c)+g(c_{(1)})\ot c_{(2)}\\ 
&(f\ot m)\De(a\ot b) -\De f(a,b)+(m\ot f)\De(a\ot b) = \\
&\phantom{(f\ot m)\De(a\ot b)} -\De (a)g(b)+g(ab)-g(a)\De (b) \nonumber \big\}
\end{align*}
and
\begin{align*}
\widehat{{\rm B}}_b^2(A)= \big\{(f,g)\;\big|\; \exists h\colon A^+\to A^+, \,\,
&f(a,b)=ah(b)-h(ab)+h(a)b, \\
   &g(c)=-c_{(1)}\ot h(c_{(2)})+\De h(c) -h(c_{(1)})\ot c_{(2)} \big\}.
\end{align*}
We view $f$ as a map from $A\ot A$ to $A$ with the understanding that
it is normalized and co-normalized in the sense that $f(\id\ot u)=0=f(u\ot\id)$ and
$\ep f=0$.  Similarly, we view $g$ as a map from $A$ to $A\ot A$ with the
understanding that it is normalized and co-normalized in the sense that $g u =0$
and $(\ep\ot\id)g=0=(\id\ot\ep)g$.
We will restrict our attention to the $A_0$-bitrivial part of the cohomology, \emph{i.~e.},
pairs $(f,g)$ where $f|_{A_0\ot A}=0=f|_{A\ot A_0}$ and $(\pi\ot\id)g=0=(\id\ot\pi)g$.
Here $\pi\colon A\to A_0$ denotes the canonical projection.
We denote the corresponding sets by $\Zbt2(A)$ and $\Bbt2(A)$,
and the relevant cohomology by
$$
\Hbt2(A)=\Zbt2(A)/\Bbt2(A).
$$
We remark that, if $A_0$ is a group algebra or the dual of a group algebra,
then we have that $\Hbt2(A)=\Hbh2(A)$, see \cite[Lemma 2.3.1]{MW}.
For $\ell\in\mathbb{Z}$, we set
$$
\Hbt2(A)_\ell=\Zbt2(A)_\ell/\Bbt2(A)_\ell,
$$
for the relevant graded cohomology. It is obtained by restricting to maps of homogeneous degree $\ell$.
We also point out the following well-known facts: if $f\colon A\ot A\to A$ is an $A_0$-trivial
Hochschild $2$-cocycle,then from the equalities $\pah f(h,x,y)=\pah f(x,h,y)=\pah f(x,y,h)$ and the $A_0$-triviality
it follows that
\begin{align*}
f(hx,y)= hf(x,y),&&
f(xh,y)= f(x,hy),&&
f(x,yh)= f(x,y)h,
\end{align*}
for all $x,y\in A$, $h\in A_0$.
If $A_0=H$ is moreover a Hopf algebra, then one has that $A\simeq B\# H$ as algebras, where $B=A^{\mathrm{co} H}$. Hence,
$f$ is $H$-stable and uniquely determined by
its values on $B\ot B$.  This follows from the fact that
$f(xh,yk)= f(x,h_{(1)}yS(h_{(2)}))h_{(3)}k$ for all
$x,y\in B$ and $h,k\in H$.  In particular,
\begin{eqnarray*}
f(hx,ky)&=& hf(x,k_{(1)}yS(k_{(2)}))k_{(3)}, \\
f(h_{(1)}x S(h_{(2)}),h_{(3)} y S(h_{(4)})) &=& h_{(1)} f(x,y) S(h_{(2)}).
\end{eqnarray*}

\subsection{Liftings as cocycle deformations}\label{subsec:partialinv} Let $H$ be a fixed Hopf algebra with bijective
antipode and $B$ a graded connected bialgebra in $\ydh$.
Set $A=B\# H$ the bosonization of $B$ with $H$.
Then $A$ is a graded Hopf algebra with $A^{\mathrm{co} H}=B$ and $A_{0}=H$, see
Subsection \ref{subsec:bosonization}.
We will study $H$-bitrivial formal deformations of $A$.
We remark that $H$-bitriviality is automatic if $H$ is a
semisimple and cosemisimple Hopf algebra such that the integrals in $H$ and $H^*$ are cocommutative, see \cite[Remark 2.3.2]{MW}.

Recall that $H$ acts on $\Hep2(B,\kk)$ by $(\eta^h)(x,y) = \eta(h_{(1)}\cdot x, h_{(2)}\cdot y)$
for all $h\in H$, $\eta \in \Hep2(B,\kk)$ and $x,y\in B$.
Let $\Hep2(B,\kk)^H$ denote the $H$-stable part.  Then we have a morphism
$\tilde{(\;)}\colon \Hep2(B,\kk)^H\to \Hep2(A,\kk)$
given by $\tilde{f}(xh,yk)=f(x,h\cdot y)\ep(k)$.  If $H$ is semisimple, then this map is an isomorphism. Moreover,
we have a connecting homomorphism \cite[Theorem 2.3.7]{MW}:
$$\pa\colon \Hep2(A,\kk)\to \Hbh2(A)$$ given by
$\pa(f)=(\pac f,0)$, where
$$
\pac f(x,y) = x_{(1)} y_{(1)} f( x_{(2)}, y_{(2)}) - f(x_{(1)}, y_{(1)}) x_{(2)}y_{(2)}.
$$
Note that for $f\in \Zep2(B,\kk)^{H}$, one has that $\pa(\tilde{f})\in \Zbt2(A)$.

\begin{defn}\label{def:M(B)} For an augmented algebra $B$ with multiplication map
$\m\colon B\ot_B B \to B$ define
$$\M(B):=\ker\big(B^+\otimes_B B^+\stackrel{\m}{\to} B^+\big).$$
\end{defn}

\begin{obs}\label{rmk:hoch-M(B)} It is known that
$ \Hep2(B,\kk) \simeq \Hom(\M(B),\kk) $, see for example \cite[Subsection 4.1]{MW}.
If $(B^+)^2$ denotes the image of the multiplication
$\m\colon B^+\ot_B B^+\to B^+$ and $\varphi\colon (B^+)^2\to B^+\ot_B B^+$
is any linear section of the multiplication map, then the isomorphism
$\Hom(\M(B),\kk)\stackrel{\sim}{\to}  \Hep2(B,\kk)$ is given by $\hat{\varphi}(f)= f_\varphi$,
where $f_\varphi(x\ot y)= f(x\ot_B y-\varphi(xy))$.
\end{obs}

Let us now fix $\ell<0$ and let $r=-\ell$.  We will show that the results and proofs from \cite{MW}
can be adjusted to show the following: if $B$ is generated as an algebra in degrees at most $r$ and
$\yd(\M(A),\mP(A))_\ell=0$, then the connecting homomorphism at degree $\ell$,
$$
\tilde{\pa}_\ell\colon \Hep2(B,\kk)^H_\ell \to \Hbt2(A)_\ell,
$$
given by $\tilde{\pa}_\ell(f) = \pa(\tilde{f})$, is surjective.
The following generalizes Lemma 4.2.2 of \cite{MW} (see also Theorem 5.3 of \cite {AKM}).
The proof is almost identical.

\begin{lema}[cf. {\cite[Lemma 4.2.2]{MW}}]\label{lem:FG} Let $A=B\#H$ be as above and
$(F,G)\in \Zbt2(A)_\ell$ for some $\ell<0$.  Set $r=-\ell$.
Assume that $B$ is generated as an algebra by elements of degree at most $r$ and that
$\yd(\M(B),\mP(B))_\ell=0$.  If $F_{r}=0$, then $(F,G)\in\Bbt2(A)$.
\end{lema}

\begin{proof} Recall the setting of Notation \ref{notation-le}.  Assume that $F_r=0$.  By degree considerations,
one has that $F_{< r}=0$, hence $F_{\le r}=0$. Analogously, one deduces that
$G_{<r}=0$. Moreover, also from degree considerations combined with that fact that $G$ is
$A_0$-cotrivial (\emph{i.~e.}, $(\pi\ot\id)G=0=(\id\ot\pi)G$), it follows that
$G_r=0$.

Now assume that for some $s>r$, we have that $(F,G)_{<s}=0$ (as observed above this holds
for $s=r+1$).  We prove that $(F,G)\sim (F',G')$ where $(F',G')_{\le s}=0$. The
conclusion of the lemma then follows by induction.

We first show that $F_s|_{M(B)}$ lies in $\yd(\M(B),\mP(B))_\ell=0$.
Let $u=\sum_i u^i\ot v^i\in (B\ot B)_s\cap M(B)$, with all $u^i, v^i$ homogeneous of strictly positive degree,
be such that $\m(u)=\sum_i u^i v^i=0$.  Then, by degree considerations and the fact that $\m(u)=0$, we have that
\begin{itemize}
\item $(F\ot\m)\De u = F(u)\ot 1$,
\item $(\m\ot F)\De u = \sum_i u^i_{(-1)}v^i_{(-1)}\ot F(u^i_{(0)}, v^i_{(0)})$,
\item $\sum_i \De(u^i)G(v ^i)=0$,
\item $G(\sum_i u^iv^i)=0$,
\item $\sum_i G(u ^i)\De(v^i)=0$.
\end{itemize}
Hence, by the compatibility of $F$ and $G$ given by $\pa^c F(u)=-\pa^h G(u)$, or, in expanded
form $(F\ot \m)\De(u)-\De F(u)+(\m\ot F)\De u = -\sum_i \De(u^i)G(v^i) +G(m(u)) -\sum_i G(u^i)\De(v^i)$), we have
$$
\De F(u) = F(u)\ot 1 + \sum_i u^i_{(-1)}v^i_{(-1)}\ot F(u^i_{(0)}, v^i_{(0)}).
$$
Therefore, we conclude that $F(u)$ is in $\mP(B)$. Since $F$ is $H$-trivial and consequently also
$H$-stable, we have that $F$ defines an element of $\yd(\M(B),\mP(B))_\ell=0$.
Hence, $F_s|_{(B\ot B)}$ is an $\ep$-coboundary (see Remark \ref{rmk:hoch-M(B)}).
Let $T\colon B\to \Bbbk$ be the map given by $F_s(x,y)=T(xy)$ for all $x,y\in B^+$.
With no loss of generality, we assume that $T=T_s$.  Since $F_s$ is $H$-stable and the $xy$'s span
$B_s$ (generation in degrees at most $r$), we have that $T$ is $H$-stable and therefore $F_s= -\pa^h T$.
Now set $(F',G')=(F,G)-\pa^b T = (F-\pa^h T, G+\pa^c T)$.  By construction, we have that $(F',G')\sim (F,G)$,
$F'_{\le s}=0$, and $G'_{<s}=0$.  We now prove that $G'_s=0$, which completes the proof.
Let $x\in A_i$, $y\in A_j$ with $i,j>0$ and $i+j=s$.  As $F_{\le s}=0$ and $G_{<s}=0$, by the compatibility
of $F$ and $G$, we get that
$
G(xy)=0
$. From this follows that $G'_s=0$, since we are assuming that $B$, and hence also $A$,
is generated as an algebra by
elements of degree at most $r$, which implies that the $xy$'s span $A_s$.
\end{proof}

\begin{obs}\label{rmk:M(B)I} Let $A=B\#H$ be as above and let $B=\T(V)/I$
for some $V\in\ydh$ and
let $I\subseteq \T(V)_{(2)}=\bigoplus_{n\ge 2} V^{\otimes n}$ be a graded ideal;
here we assume that $B$ is graded by $\deg(v)=1$ for $v\in V$ with $v \neq 0$ (and hence generated as an algebra in degree $1$).
Then
$$\M(B)\simeq I/\big(T(V)^+ I + I T(V)^+\big).$$
Furthermore, if $I$ is generated by $R\in \ydh$, then the canonical map $R\to \M(B)$ is
an epimorphism of graded $\ydh$ modules (see \cite[Corollary 4.5]{AKM}). If $R$ is minimal
(in the sense that $R\cap (T(V)^+ I+I T(V)^+)=0)$,
then the map in question is an isomorphism.
In particular:
\begin{enumerate}
\item If $\yd(R, \mP(B))_\ell=0$ for some $\ell<0$,
then every cocycle pair $(F,G)\in\Zbt2(A)_\ell$
satisfying $F_{-\ell}=0$ is cohomologically trivial.
\item If $B=\toba(V)=\T(V)/\langle R\rangle$ is a Nichols algebra and $\yd(R,V)=0$, then every cocycle pair
$(F,G)\in \Zbt2(A)_\ell$, $\ell<0$, satisfying $F_{-\ell}=0$ is cohomologically trivial.
\end{enumerate}
\end{obs}

From now on we assume that there exists a right integral
$\lambda\colon H\to\kk$ for $H^*$ (\emph{i.~e.}, $\id*\lambda = \lambda$) such that
\begin{enumerate}
\item $\lambda(1)=1$,
\item $\lambda$ is stable under the adjoint action of $H$.
\end{enumerate}
Note that such a map $\lambda$ always exists if $H$ is a cosemisimple unimodular Hopf
algebra, \emph{i.~ e.}, the space of right and left integrals in $H$ coincide.
In case $H$ is finite-dimensional (we are not assuming this anywhere in this section), the later holds if and only if $H$ is semisimple.

\begin{lema}\label{lem-lambda}  Let $(F,G)\in \Zbt2(A)_\ell$ for some $\ell<0$ and write $r=-\ell$.
Then the map $f\colon B\ot B\to\kk$ given by $f|_{(B\ot B)_r}=\lambda\circ F$ and $f|_{(B\ot B)_s}=0$
for $s\not=r$, is an $H$-stable $\ep$-cocycle.  Moreover, if $\tilde{f}\colon A\ot A\to\kk$ is the
induced $\ep$-cocycle, then $(F+\pac \tilde{f})_r=0$.
\end{lema}

\pf
Since $F$ is $H$-stable and $\lambda$ is stable under the adjoint action of $H$, we have that
$f$ is $H$-stable. Let $x,y$ be homogeneous elements of $B$ of strictly positive degrees such that
$x\ot y\in (B\ot B)_s$.  As in the proof of Lemma \ref{lem:FG}, we have that $F_{<r}=0$ and $G_{\le r}=0$
(by using degree considerations
and the fact that $G$ is $H$-cotrivial).  Hence $\pa^h G(x,y)=0$,
and the compatibility between $F$ and $G$ yields that
$$\De F(x,y)= x_{(1)}y_{(1)}\ot F(x_{(2)},y_{(2)})+ F(x_{(1)},y_{(1)})\ot x_{(2)}y_{(2)} = x_{(-1)}y_{(-1)}\ot F(x_{(0)},y_{(0)})+F(x,y)\ot 1.$$
Applying $\m(\id\ot\lambda)$ to this equality gives
$
(\id*\lambda)(F(x,y)) = x_{(-1)}y_{(-1)}f(x_{(0)},y_{(0)})+F(x,y)
$.
Since $\id*\lambda=\lambda$, we get that $F(x,y)=f(x,y)-x_{(-1)}y_{(-1)}f(x_{(0)},y_{(0)})$.  Moreover, as
$\tilde{f}|_{B\ot B}=f|_{B\ot B}$ and $f|_{<r}=0$, we have that $(F+\pac\tilde f)(x,y)=0$.
Since $F$ and $\pac\tilde f$ are uniquely determined by their values on $B\ot B$, we conclude that
$(F+\pac\tilde f)_r=0$.
\epf

We end this subsection with the proof of Theorem \ref{thm:M(B)V} and two technical remarks which are useful for calculations.

\smallbreak
\noindent \textit{Proof of Theorem A.} Let $(F,G)\in \Zbt2(A)_\ell$ for some $\ell<0$ and let $r=-\ell$ as above.  Let $f\colon B\ot B\to\kk$ be the cocycle
from Lemma \ref{lem-lambda} and $\tilde{f}\colon A\ot A\to\kk$ the induced cocycle on $A$.
Set $(F',G)=(F,G)+\tilde{\pa}_\ell f = (F'+\pac\tilde{f},G)$.
By Lemma \ref{lem-lambda}, it follows that $F'_{r}=0$. Hence, by Lemma \ref{lem:FG}
we have that $(F',G)\in\Bbt2(A)$.
Thus, the cohomology classes of $(F,G)$ and $-\tilde{\pa}_\ell f$ coincide and
consequently, the cohomology class of $(F,G)$ is in the image of $\tilde{\pa}_\ell$.
\qed

\begin{obs}\label{rmk:formula-def-gen}
Fix $\ell<0$ and set $r=-\ell$.
For $f\in\Zep2(B,\kk)^H_\ell$ and $(u,v)\in (B^+\ot B^+)_r$, it holds
$$\pac (\tilde f)(u,v)= u_{(-1)}v_{(-1)}f(u_{(0)},v_{(0)})- f(u,v).$$
Indeed, assume $u\in B_{k}$ and $v\in B_{\ell}$ with $k+\ell = r$.
Since $B = A^{\co\pi}$, we may assume that $u\in P_{k}(B)$ and $v\in P_{\ell}(B)$.
Then, $\com(u) \in u\ot 1 + u_{(-1)}\ot u_{(0)} + \sum_{i=1}^{k-1} A_{i}\ot A_{k-i}$
and $\com(v) \in v\ot 1 + v_{(-1)}\ot v_{(0)} + \sum_{i=1}^{\ell-1} A_{i}\ot A_{\ell-i}$. This implies that
$$
\pac(\tilde f)(u,v)  = u_{(1)}v_{(1)}f(u_{(2)},v_{(2)}) - f(u_{(1)},v_{(1)})u_{2}v_{2} \\
= u_{(-1)}v_{(-1)}f(u_{(0)},v_{(0)}) - f(u,v).
$$
\end{obs}

\begin{obs}\label{rmk:formula-def-part}
Fix $\ell<0$, set $r=-\ell$ and let $f\in\Zep2(B,\kk)^H_\ell$. If $B=T(V)/\langle R\rangle$ with $R \subseteq V\ot V$, then the only possible candidate for a lifting
corresponding to $f$ is
$$A_f = (\T(V)\#H)/\langle R_f \rangle,$$
where $R_f$ is obtained from $R$ by replacing every relation
$x=\sum_i u^i\ot v^i \in R\subseteq V\ot V$ by
$$x_f = x + \pac(\tilde f)(x)= x - \sum_i \big(f(u^i_{(1)},v^i_{(1)})u^i_{(2)}v^i_{(2)}- u^i_{(1)}v^i_{(1)}f(u^i_{(2)},v^i_{(2)})\big).$$
This follows from the fact that $m_{A_{f}}|_{(A\ot A)_2} = m_{A} + \pac(\tilde f)$.
Note that if $\deg(x)<r$, then $x_f=x$, and in case $\deg(x)=r$, it holds that
$$
x_f = x - \sum_i \big(f(u^i,v^i) - u^i_{(-1)}v^i_{(-1)}f(u^i_{(0)},v^i_{(0)})\big).
$$
In principle, it is hard to know if the above is indeed a lifting.
This is guaranteed if there is a multiplicative cocycle with infinitesimal part equal to $\tilde{f}$.
In the very special case when the braiding is symmetric (and hence $R$ is in degree $2$,
\emph{i.~e.}, $R\subseteq V\ot V$), this is always the case as such cocycles are produced by $e^{\tilde{f}}$.
\end{obs}

\subsection{Cohomologically homogeneous liftings}
The following results in this subsection are well-known for graded deformations of algebras \cite{BG} and it is to some extent implicit in \cite{DCY}.
We include some sketches of proofs for completeness.

\begin{lema}
Let $B=(B,m,\Delta)$ be a graded bialgebra and let $U=(B[t],m_t^U,\Delta_t^U)$ be a graded deformation.
If $\ell\in\mathbb{N}$ and $f\in \widehat{\mathrm{B}}^2_b(B)_{-\ell}$, then $U$ is equivalent to a deformation $V=(B[t],m_t^V,\Delta_t^V)$ such that
for $k<\ell$ we have $(m_k^U,\Delta_k^U)=(m_k^V,\Delta_k^V)$ and $(m_\ell^U,\Delta_\ell^U)=(m_\ell^V,\Delta_\ell^V)+f$.
Moreover, we can additionally assume that there is an isomorphism $\Phi\colon U\to V$ such that for $k<\ell$ we have that $\Phi|_{B_k}=\id|_{B_k}$.
\end{lema}

\begin{proof} Let $s\colon B\to B$ be the homogeneous map of degree $-\ell$ such that $f=\pab s$ and consider
the $\kk[t]$-linear map $\Phi\colon B[t]\to B[t]$ given
by $\Phi|_{B}=\id+s t^k$.  The lemma then follows by setting $V$ to be the unique bialgebra structure on $B[t]$ such that $\Phi$ is a
bialgebra isomorphism, \emph{i.~e.}, $m_t^V=\Phi^{-1}\circ m_t^U\circ (\Phi\otimes \Phi)$ and $\Delta_t^V=(\Phi^{-1}\otimes\Phi^{-1})\circ \Delta_t^U\circ \Phi$.
\end{proof}

\begin{lema}[Infinitesimal cohomological difference] Let $B$ be a graded bialgebra and let
$U=(B[t],m_t^U,\Delta_t^U)$ and $V=(B[t],m_t^V,\Delta_t^V)$ be non-equivalent graded deformations.
Then $V$ is equivalent to a deformation $W=(B[t],m_t^W,\Delta_t^W)$ such that for some $\ell\in\mathbb{N}$ we have that
$(m_k^U,\Delta_k^U)=(m_k^W,\Delta_k^W)$ for $k<\ell$ and $(m_\ell^U,\Delta_\ell^U)-(m_\ell^W,\Delta_\ell^W)\in \widehat{\mathrm{Z}}_b^2(B)_{-\ell}$
represents a non-trivial cohomology class.  The number $\ell$ and the cohomology class of $(m_\ell^U,\Delta_\ell^U)-(m_\ell^W,\Delta_\ell^W)$
is uniquely determined by the equivalence classes of $U$ and $V$.
\end{lema}

\begin{proof}
If $(m_1^U,\Delta_1^U)-(m_1^V-\Delta_1^V)$ represents a nontrivial cohomology class, then we are done.
Otherwise, let $s_1\colon B\to B$ be a homogeneous map of degree $-1$ such that $(m_1^U,\Delta_1^U)-(m_1^V-\Delta_1^V)=\pab s_1$
and let $\Phi_1\colon B[t]\to B[t]$ be given by $\Phi_1=\id+s_1 t$.  Take  $W_1$ as the unique bialgebra structure on
$B[t]$ such that $\Phi_1$ is a bialgebra isomorphism; then $(m_1^U,\Delta_1^U)=(m_1^{W_1}-\Delta_1^{W_1})$.
If $(m_2^U,\Delta_2^U)-(m_2^{W_1}-\Delta_2^{W_1})$ represents a nontrivial cohomology class, then we are done. On the contrary,
take a homoheneous map $s_2\colon B\to B$  of defree $-2$ such that $(m_2^U,\Delta_2^U)-(m_2^{W_1}-\Delta_2^{W_1})=\pab s_2$,
and define $\Phi_2\colon B\to B$ by $\Phi_2|_B=\id+s_1 t+s_2 t^2=\Phi_1+t^2 s_2$. Let $W_2$ be
the unique bialgebra structure on $B[t]$ such that $\Phi_2$ is a bialgebra isomorphism.  In a similar fashion as above,
we may define $s_k, \Phi_k$ for $k\geq 3$.  We claim that
this process eventually stops.  Indeed, suppose that it does not.  We then define $W=\lim_{k\to\infty} W_k$
to be the unique bialgebra structure on $B[t]$ such that $\Phi\colon B[t]\to B[t]$ given by
$\Phi|_B=\id+\sum_{k=1}^\infty s_k = \lim_{k\to\infty} \Phi_k$ is a bialgebra isomorphism.  Since the sum $\sum_{k=1}^\infty s_k$
is locally finite it is therefore well-defined. But in such a case, we would have that $U=W$;  a contradiction.
\end{proof}

\begin{defn}[Infinitesimal difference] Let $B$ be a graded bialgebra and $U=(B[t],m_t^U,\Delta_t^U)$,
$V=(B[t],m_t^V,\Delta_t^V)$ be non-equivalent graded deformations.  Let $W$ and $\ell$
be as in the lemma above.  We define the \textit{infinitesimal difference} $U-_{\varepsilon}V$ to be the cohomological class
of $(m_\ell^U,\Delta_\ell^U)-(m_\ell^W,\Delta_\ell^W)$.  In such a case, we say that the infinitesimal difference is of degree $\ell$.
If $U$ and $V$ are equivalent, then we define $U-_\varepsilon V=0$ and say that
$\deg(U-_{\varepsilon} V)=\infty$.
If $L_1, L_2$ are liftings of $B$, we define their infinitesimal difference $L_1-_\varepsilon L_2$ to be the
infinitesimal difference between the associated graded deformations.
\end{defn}

Note that infinitesimal difference is ``associative" in each degree in the following sense:
if $\deg(L-_\varepsilon S)=\deg(S-_\varepsilon T)=\ell$, then
\begin{enumerate}
\item[$(a)$]
$\deg(L-_\varepsilon T)=\ell$ and $L-_\varepsilon T=(L-_\varepsilon S)+(S-_\varepsilon T)$, if $L-_\varepsilon S\not=-(S-_\varepsilon T)$,
\item[$(b)$] $\deg(L-_\varepsilon T)>\ell$, if  $L-_\varepsilon S=-(S-_\varepsilon T)$.
\end{enumerate}

\begin{defn} Let $(L_\lambda)_{\lambda\in\Lambda}$ be a collection of liftings of a graded bialgebra $B$.
We say that the collection is \textit{cohomologically homogeneous}, if for every $\lambda\in\Lambda$,
every $\ell\in\mathbb{N}$, and every cohomology class $\alpha\in\widehat{\mathrm{H}}_b(B)_{-\ell}$,
there exists a $\mu\in\Lambda$ such that $L_\lambda-_\varepsilon L_\mu = \alpha$.
\end{defn}

\begin{defn} We say that the lifting problem for a graded bialgebra $B$ is \textit{obstruction-free} if every partial
graded deformation (\emph{i.~e.}, deformation over $\kk[t]/(t^\ell)$ for some $\ell\in\mathbb{N}$) extends to a formal graded deformation.
\end{defn}

We end this section with the following theorem.

\begin{theorem} Let $B$ be a graded bialgebra such that $\widehat{\mathrm{H}}_b^2(B)_{-}:=\bigoplus_{k\in\mathbb{N}} \widehat{\mathrm{H}}_b^2(B)_{-k}$
is finitely graded, \emph{i.~e.}, for sufficienly large $k\in\mathbb{N}$ we have $\widehat{\mathrm{H}}_b^2(B)_{-k}=0$.  Then:
\begin{enumerate}
\item[(i)] The lifting problem for $B$ is obstruction-free if and only if there exists a cohomologically
homogeneous collection of liftings.
\item[(ii)] If $(L_\lambda)_{\lambda\in\Lambda}$ is a cohomologically homogeneous collection of liftings,
then, up to equivalence, the collection contains all liftings.
\end{enumerate}
\end{theorem}

\begin{proof} The only (somewhat) non-obvious part of the theorem is assertion (ii). Suppose on the contrary,
that $(L_\lambda)_{\lambda\in\Lambda}$ is a cohomologically homogeneous collection of liftings such that there exists a
lifting $L$ that is not equivalent to any $L_\lambda$.  Let $\ell\in\mathbb{N}$ be the largest number such
that $L-_\varepsilon L_\lambda=:\alpha$ is of degree $\ell$.  By cohomological homogeneity,
there exists $\mu\in\Lambda$ such that $L_\lambda-_\varepsilon L_\mu=-\alpha$.  But then we either
have that $L$ and $L_\mu$ are equivalent or that the cohomological difference $L-_\varepsilon L_\mu$
is of degree strictly larger then $\ell$;  a contradiction.
\end{proof}

\section{Copointed Hopf algebras over dihedral groups} \label{sec:pointed-dihedral}

\subsection{The commutative algebra $\Bbbk^{\dm}$}
As in \cite{FG}, we use the following
presentation by generators and relations
for the dihedral group of order $2m$:
\begin{equation}\label{eq:def-dihedral}
\mathbb D_m:=\langle g,h| \ g^2=1=h^m  \,\, , \,\,
gh=h^{-1}g \rangle.
\end{equation}

Because of our purposes, we
assume that $m=4a \geq 12$, $n=\frac{m}{2}= 2a$ and we fix
$\omega$ an $ m $-th primitive root of unity.

A linear basis of the group algebra $\Bbbk \dm$ is given by
$\{e_{ij}=g^{i}h^{j}:\ i=0,1,\ j=0,\ldots, m-1\}$. Let
$\{\varphi_{ij}:\ i=0,1,\ j=0,\ldots, m-1\}$ be the corresponding dual basis,
\emph{i.~e.}, $\varphi_{i,j}(e_{k,\ell})=\delta_{i,k}\delta_{j,\ell}$.
By definition, we have that
 \begin{itemize}
 \item $\varphi_{ij}\varphi_{k\ell} = \delta_{ik}\delta_{j\ell}\varphi_{ij}$.
\item  $1_{\Bbbk^{\dm}}=\sum_{i,j}\varphi_{ij}$.
 \item $\eps(\varphi_{ij})=\delta_{i,0}\delta_{j,0}$.
\item $\com(\varphi_{ij}) = \sum_{k,\ell} \varphi_{i+k,(-1)^{k}(j-\ell)}\ot \varphi_{k,\ell}$.
\item $\cS(\varphi_{ij})=\varphi_{i,(-1)^{i+1}j}$.
 \end{itemize}
\emph{Throughout this section we set $H = \Bbbk^{\dm}$}. The group
$G(H)\simeq \Z_{2}\times \Z_{2}$ is given by $\Alg(\Bbbk\dm,\Bbbk)$,
which coincides with the group of multiplicative characters $\widehat{\dm}$ given
by the one-dimensional representations of $\dm$. These are given by the following table:
\begin{table}[h]\label{table:chardm}
\begin{center}
\begin{tabular}{|c|c|c|c|c|}
\hline $z$&  $h^n$ & $h^b$, $1\leq b\leq n-1$ & $g$ & $gh$
\\ \hline  \hline $\alpha_0$ &1&1&1&1
\\ \hline $\alpha_1$ &1&1&$-1$&$-1$
\\ \hline $\alpha_2$ &$(-1)^n$& $(-1)^b$ &1&$-1$
\\ \hline $\alpha_3$ &$(-1)^n$& $(-1)^b$ &$-1$&1\\
\hline
\end{tabular}
\end{center}
\caption{Linear characters of $\dm$}
\end{table}

\noindent For $0\leq k\leq 3$, write $\alpha_{k} = \sum_{i,j} \alpha_{k}(g^{i}h^{j})\varphi_{ij}$; explicitly,
\begin{align*}
 \alpha_{0}&=
\sum_{i,j} \varphi_{ij}
 = \sum_{j} \varphi_{0j} +  \sum_{j} \varphi_{1j}=\eps,\\
 \alpha_{1} &=
 \sum_{i,j} (-1)^{i}\varphi_{ij}= \sum_{j} \varphi_{0j} -  \sum_{j} \varphi_{1j},\\
 \alpha_{2} &=
 \sum_{i,j} (-1)^{j}\varphi_{ij}
 = \sum_{j} (-1)^{j}\varphi_{0j} +  \sum_{j} (-1)^{j}\varphi_{1j},\\
 \alpha_{3} &=
 \sum_{i,j} (-1)^{i+j}\varphi_{ij}
 =\sum_{j} (-1)^{j}\varphi_{0j} -  \sum_{j} (-1)^{j}\varphi_{1j}.
\end{align*}
It holds that $\alpha_{2}\alpha_{3}=\alpha_{1}=\eps$, $\ord \alpha_{i} = 2$ with $i\neq 0$
and $G(\Bbbk^{\dm}) = \langle \alpha_{2}\rangle \times \langle \alpha_{3}\rangle \simeq \Z_{2} \times \Z_{2}$.
For $k=0,1$ and $0\leq r\leq n$,
write $\theta_{k,r} = \sum_{\ell} \omega^{r\ell}\varphi_{k,\ell}\in \Bbbk^{\dm}$. It holds that
\begin{align}
1 = \eps = \theta_{0,0} + \theta_{1,0},\quad
 \alpha_{1} &= \theta_{0,0} - \theta_{1,0},\quad \alpha_{2} = \theta_{0,n} + \theta_{1,n},\quad
 \alpha_{3} =\theta_{0,n} - \theta_{1,n}. \label{eq:theta1}\\
 \theta_{0,r}\theta_{1,s} &=0,\quad  \theta_{k,r}\theta_{k,s} =\theta_{k,r+s},\quad
 \text{ for all }k=0,1,\ 0\leq r,s \leq n.\label{eq:theta3}
\end{align}

\subsection{Yetter-Drinfeld modules and finite-dimensional Nichols algebras
over $\Bbbk^{\dm}$}\label{subsec:ydmodkdm}
As we pointed out before, the category $\ydhs$  is braided equivalent to $\ydh$.
Since finite-dimensional Nichols algebras
in $\yddm$ were classified in \cite{FG}, all finite-dimensional
Nichols algebras in $\ydkm$ are known. For details concerning
simple objects in $\yddm$ see \cite[Section 2]{FG}.

The irreducible Yetter-Drinfeld modules that give rise
to finite-dimensional Nichols algebras are associated with
the conjugacy classes of $h^{n}$ and $h^{i}$ with $1\leq i < n$, see
\cite[Table 2]{FG}. In the following, we describe these modules explicitly as well as
the families of reducible Yetter-Drinfeld modules with finite-dimensional Nichols
algebras associated with them.

Recall that the non-trivial conjugacy classes of $ \dm $ and the corresponding
centralizers are
\begin{itemize}
\item $\Oc_{h^{n}}= \{h^{n}\}$ and $C_{\dm}(h^{n})= \dm$.
\item $\Oc_{h^{i}}= \{h^{i}, h^{m-i}\}$ and $C_{\dm}(h^{i})=
\langle
h \rangle \simeq  \Z/(m)$, for $1\leq i <n$.
\item $\Oc_{g}= \{gh^{j}: j\text{ even}\}$ and $ C_{\dm}(g) =
\langle g \rangle \times \langle h^{n}\rangle \simeq \Z/(2) \times \Z/(2) $.
\item$\Oc_{gh}= \{gh^{j}:j\text{ odd}\}$ and
$ C_{\dm}(gh) =
\langle gh \rangle \times \langle h^{n}\rangle \simeq \Z/(2) \times \Z/(2) $.
\end{itemize}

\subsubsection{Yetter-Drinfeld modules and Nichols algebras associated with $\oc_{h^i}$, with $1\leq i <n$}
\label{subsubsec:ydmodmike}
For $0\leq k <m$, denote by
$\Bbbk_{\chi^{k}}$ the simple representation of $C_{\dm}(h^{i}) = \langle
h \rangle \simeq  \Z/(m)$
given by the character $\chi^{k}(h) = \omega^{k}$.

Take $e$ and $g$ as representatives of left cosets in $\dm / \langle h \rangle$,
with $h^{i} = eh^{i}e$ and $h^{m-i} = gh^{i}g$.
Then $M_{(i,k)} = M(\oc_{h^i},\chi^{k}) \in \ydkm$  is
spanned linearly by the elements
$y^{(i,k)}_{1} = e\ot 1$ and $y^{(i,k)}_{2} = g\ot 1$.
Its Yetter-Drinfeld module structure is given by
\begin{align*}
\varphi_{rs}\cdot y^{(i,k)}_{1}  &= \varphi_{rs}(\cS(h^{i})) y^{(i,k)}_{1}
= \delta_{r,0}\delta_{s,-i}\, y^{(i,k)}_{1},  \quad &
\varphi_{rs}\cdot y^{(i,k)}_{2}  &= \varphi_{rs}(\cS(h^{-i})) y^{(i,k)}_{2}
=\delta_{r,0}\delta_{s,i}\, y^{(i,k)}_{2},\\
\nonumber \lambda(y^{(i,k)}_{1}) &  =  \theta_{0, -k} \ot y^{(i,k)}_{1}
+ \theta_{1,k} \ot y^{(i,k)}_{2},\quad &
\lambda(y^{(i,k)}_{2}) &  =  \theta_{1,-k} \ot y^{(i,k)}_{1}
+ \theta_{0,k} \ot y^{(i,k)}_{2}.
\end{align*}
The irreducible modules with finite-dimensional Nichols algebra
are the ones given by the pairs $(i,k)$  satisfying that $\omega^{ik}=-1$.
We set $J=\{(i,k):\ 1\leq i< n, 1\leq k < m \text{ such that }\omega^{ik}=-1\}$.
By
 \cite[Theorem 3.1]{AF-alt-die}, one has
$\toba(\oc_{h^i},  \chi^{k}) \simeq \bigwedge
M_{(i,k)}$, for all $(i,k)\in J$, and
 $\dim \toba(\oc_{h^i},
\chi^{k})= 4$.

\begin{exa}
Consider $M_{i,n}=\Bbbk\{a_{1},a_{2}\} \in  \ydkm$ with $i$ odd and $1\leq i \leq n-1$.
If we take $x=a_{1}+a_{2}$ and $y=a_{1}-a_{2}$ we have that
\begin{align*}
\lambda(x) & = \theta_{0,n} \ot a_{1} + \theta_{1,n} \ot a_{2}  +
\theta_{1,n} \ot a_{1} + \theta_{0,n} \ot a_{2}
= (\theta_{0,n}+\theta_{1,n})\ot (a_{1} + a_{2}) = \alpha_{2}\ot x,\\
\lambda(y) & = \theta_{0,n} \ot a_{1} + \theta_{1,n} \ot a_{2}  -
\theta_{1,n} \ot a_{1} - \theta_{0,n} \ot a_{2}
= (\theta_{0,n}-\theta_{1,n})\ot (a_{1} - a_{2}) = \alpha_{3}\ot y,
\end{align*}
that is, both elements are homogeneous with respect to the coaction of $\Bbbk^{\dm}$ on $M_{i,n}$. Thus, we may present $M_{i,n} \in \ydkm$ as the
Yetter-Drinfeld modules spanned linearly by $x$ and $y$ with its structure given by
\begin{align*}
\lambda(x) & = \alpha_{2}\ot x,\qquad\qquad \qquad \lambda(y) = \alpha_{3}\ot y,\\
 \varphi_{k\ell}\cdot x & = \varphi_{k\ell}\cdot (a_{1} + a_{2}) =  \delta_{k,0}\delta_{i,-\ell}a_{1} +
 \delta_{k,0}\delta_{i,\ell}a_{2}=
 \delta_{k,0}(\delta_{i,-\ell}\frac{x+y}{2}
+ \delta_{i,\ell}\frac{x-y}{2}),\\
\varphi_{k\ell}\cdot y & = \varphi_{k\ell}\cdot (a_{1} - a_{2})  = \delta_{k,0}\delta_{i,-\ell}a_{1} -
 \delta_{k,0}\delta_{i,\ell}a_{2}=
\delta_{k,0}(\delta_{i,-\ell}\frac{x+y}{2}
- \delta_{i,\ell}\frac{x-y}{2}),\\
\alpha_{2}\cdot x &= -x,\qquad\qquad\qquad \alpha_{3}\cdot x = -x,	\\
\alpha_{2}\cdot y &= -y,\qquad\qquad\qquad\alpha_{3}\cdot y = -y.
\end{align*}
\end{exa}

Consider now the set
$\II$  of all sequences of finite length of lexicographically ordered pairs
$((i_{1},k_{1}),\ldots,(i_{r},k_{r}))$ such that  $(i_{s},k_{s}) \in J$
and $\omega^{i_{s}k_{t}+ i_{t}k_{s}} = 1$ for
all $1\leq s,t\leq r$.

For $I = ((i_{1},k_{1}),\ldots,(i_{r},k_{r})) \in \II$, we define $M_{I}=
\bigoplus_{1\leq j\leq r}M_{(i_{j},k_{j})}$.
By \cite[Proposition 2.5]{FG}, we have that
$\toba(M_{I})\simeq \bigwedge M_{I}$  and
$\dim\toba(M_{I}) = 4^{|I|}$, where $|I|=r$ denotes the length of $I$.

\subsubsection{Yetter-Drinfeld modules and Nichols algebras associated with $\oc_{h^n}$}
\label{subsubsec:ydmodmel}
Since
$h^n$ is central, $\oc_{h^n}=\{h^n\}$ and $C_{\dm}(h^n)=\dm$.
The irreducible representations of $\dm$ are
of degree 1 or 2. Explicitly, there are:
\begin{itemize}
  \item[(i)] $n-1=\frac{m-2}{2}$ irreducible representations of degree 2 given by
$\rho_\ell:\dm\to\GL(V)$ with
\begin{align}\label{eq:repdeg2}
\rho_\ell (g^ah^b)=\begin{pmatrix} 0 & 1 \\
                  1 & 0 \\
                \end{pmatrix}^a
 \begin{pmatrix}
                  \omega^\ell & 0 \\
                  0 & \omega^{-\ell} \\
                \end{pmatrix}^b, \quad \ell\in \N\text{ odd with }1\leq \ell<n.
\end{align}
\item[(ii)] 4 irreducible representations of degree 1 given in Table \ref{table:chardm}.
\end{itemize}

The irreducible Yetter-Drinfeld modules with finite-dimensional Nichols algebra
are the ones given by the two-dimensional representations
$\rho_{\ell}$ with $\ell\in \N$ odd.

Fix $\ell \in \N$ odd with $1\leq \ell <n$ and consider
the two-dimensional simple representation $(\rho_{\ell}, V)$ of
$\dm$ described in \eqref{eq:repdeg2}
above.
Then
$M_{\ell} = M(\oc_{h^n},\rho_\ell)  \in \ydkm$ is
spanned linearly by the elements
$ x^{(\ell)}_{1}, x^{(\ell)}_{2}$, and
its Yetter-Drinfeld module structure is given by
\begin{align}\label{eq:dualmel}
\varphi_{ij}\cdot x^{(\ell)}_{k}  &=  \varphi_{ij}(\cS(h^{n}))x^{(\ell)}_{k}
= \delta_{i,0}\delta_{j,n}x^{(\ell)}_{k},  \quad k=1,2 & \text{ and }\\
\nonumber \lambda(x^{(\ell)}_{1}) &  =  \theta_{0,-\ell} \ot x^{(\ell)}_{1}
+ \theta_{1,\ell} \ot x^{(\ell)}_{2},\quad &
\lambda(x^{(\ell)}_{2}) &  =  \theta_{1,-\ell} \ot x^{(\ell)}_{1}
+ \theta_{0,\ell} \ot x^{(\ell)}_{2}.
\end{align}
By
 \cite[Theorem 3.1]{AF-alt-die}, one has that
$\toba(\oc_{h^n}, \rho_{\ell}) \simeq \bigwedge
M_{\ell}$, and consequently
 $\dim \toba(\oc_{h^n},
\rho_{\ell})= 4$.

Consider now the set
$\Lc$  of all sequences of finite length $(\ell_{1},\ldots,\ell_{r})$ with
$ \ell_{i} \in \N$  odd  and
$1\leq \ell_{1},\ldots, \ell_{r}<n$.
Then, for $L = (\ell_{1},\ldots,\ell_{r}) \in \Lc$ we define
$M_{L} = \bigoplus_{1\leq i \leq r} M_{\ell_{i}}$.
Clearly, $M_{L} \in \yddm$ is \emph{reducible} and
by \cite[Proposition 2.8]{FG}, we have that
$\toba(M_{L})\simeq \bigwedge M_{L} $ and $\dim
\toba(M_{L})=4^{|L|}$, where $|L|=r$ denotes the length of $L$.

\subsubsection{Yetter-Drinfeld modules
and Nichols algebras associated with mixed classes}\label{subsubsec:ydmodmelandmike}
Finally, we describe a family of reducible Yetter-Drinfeld modules
given by direct sums of the modules described above.

Let $\K$ be the set of all pairs of sequences of finite length
$(I,L)$ with $I= ((i_{1},k_{1}),\ldots,(i_{r},k_{r})) \in \II$ and
$L = (\ell_{1},\ldots,\ell_{s})\in \Lc$ such that $k_{j}$ is odd
for all $1\leq j\leq r$ and $\omega^{i_{j}\ell_{t}}=-1$
for all $1\leq j\leq r$ and $1\leq t\leq s$.

As before, for $(I,L) \in \K$ we define $M_{I,L} =
\left(\bigoplus_{1\leq j\leq s} M_{(i_{j},k_{j})}\right)\oplus
\left(\bigoplus_{1\leq t\leq s} M_{\ell_{t}}\right) $.
By \cite[Proposition 2.12]{FG}, we have that
$\toba(M_{I,L}) \simeq \bigwedge M_{I,L}$ and $\dim \toba(M_{I,L})
= 4^{\vert I \vert + |L|} $.

We end this subsection with the classification of
finite-dimensional Nichols algebras over $\kdm$. It follows from
the braided equivalence between $\yddm$ and $\ydkm$.
\begin{theorem}\cite[Theorem A]{FG}\label{thm:all-nichols-dm}
Let $\toba(M)$ be a finite-dimensional Nichols algebra
in $ \ydkm $. Then $\toba(M)\simeq \bigwedge M$, with $M$
isomorphic
either to $M_{I}$, or to
$M_{L}$, or to $M_{I,L}$, with $I\in \II$, $L\in \Lc$ and
$(I,L) \in \K$, respectively.\qed
\end{theorem}

\medbreak

\subsection{Finite-dimensional copointed Hopf algebras
over ${\dm}$}\label{subsec:hadm}
In this section we describe
all finite-dimensional
Hopf algebras $A$ such that the corradical
$A_{0}$ is isomorphic to
$\Bbbk^{\dm}$,
for $m=4a \geq 12$. Note that these algebras
have the Chevalley property.

Using again the braided equivalence $\yddm\cong \ydkm$, we have the following result.

\begin{theorem}\cite[Theorem 3.2]{FG}\label{thm:item-c}
Let $A$ be a finite-dimensional Hopf algebra with
$A_{0} = \Bbbk^{\dm}$. Then $A$ is generated in degree one, that is, by its first term
$A_{1}$ in the coradical filtration.\qed
\end{theorem}

\begin{obs}\label{rmk:kdm-inv}
To find all possible Hopf algebras $A$ with $A_{0}=\kk^{\dm}$ we will use the
theory developed in Section \ref{sec:cohomology}. Assume $\gr A \simeq B\# \kk^{\dm}$ with
$B\in \ydkm$.
Let $b \in B$ and denote
by $\lambda(b) = b_{(-1)}\ot b_{(0)}$ the coaction of $b \in \yddm$.
Recall that the action of $\kk^{\dm}$ on $B$ is given by
$f\cdot b = f(\cS(b_{(-1)}))b_{(0)}$. By the definition of the $\kk^{\dm}$-action on
$ \Zep2(B,\kk)$ given in Subsection \ref{subsec:partialinv}, we have that a $2$-cocycle
$\eta \in \Zep2(B,\kk)$ is $\kk^{\dm}$-invariant  if and only if
$$f(\cS(y_{(-1)}x_{(-1)}))\eta(x_{(0)},y_{(0)}) = f(1)\eta(x,y)\qquad\text{
for all }f\in \kk^{\dm},\ x,y \in B.$$
\end{obs}

As said before, any $2$-cocycle deformation of a coradically graded Hopf algebra $A$ is a lifting.
In general, the converse is not known to be true. In case the coradical is $\kk^{\dm}$,
the converse holds by the following result.

\begin{theorem}\label{thm:defcocycles}
Let $A$ be a finite-dimensional Hopf algebra with $A_{0}\simeq \kk^{\dm}$. Then $A$ is a cocycle
deformation of $\toba(M)\# \kk^{\dm}$ for some finite-dimensional
Nichols algebra $\toba(M) \in \ydkm$.
\end{theorem}

\pf
Since the coradical $A_{0}$ is a Hopf subalgebra, we know that $\gr A \simeq B\#\kk^{\dm}$
for some finite-dimensional braided Hopf algebra $B\in \ydkm$,  and $A$ is a lifting of $B\#\kk^{\dm}$.
Moreover, by Theorem \ref{thm:all-nichols-dm},
$B\simeq \toba(M)$ for some $M\in \ydkm$.
By Remark \ref{rmk:lifting=def} and Theorem \ref{thm:M(B)V}, $A$ then corresponds to a formal deformation with
infinitesimal part given by $\partial^c(\tilde\eta)$, for some
$\eta \in \Zep2(B,\kk)^{\kk^{\dm}}$. On the other hand, since the braiding in
$\ydkm$ is symmetric, by \cite[Corollary 2.6]{GM},
the map $\sigma = e^{\tilde{\eta}}
=\sum_{i=0}^\infty \frac{\tilde{\eta}^{* i}}{i!}\colon \gr A\otimes \gr A\to \kk
$ is a normalized multiplicative $2$-cocycle with infinitesimal part equal to $\tilde{\eta}$. Hence, by
Remark \ref{rmk:formula-def-part}, we have that $(\gr A)_{\sigma}\simeq A$ because $B$ is quadratic.
\epf

\begin{obs}\label{rmk:H2BM}
Let $B=\toba(V)$ be a finite-dimensional Nichols algebra
in $ \ydkm $, $\{x_{j}\}_{j\in J}$ a basis of $V$ and
$\{d_{j}\}_{j\in J}$ its dual basis.
Since $B$ is an exterior algebra,
by Remark \ref{rmk:M(B)I} we have that
$M(B)$ is linearly generated by $\{x_{j}\ot x_{k} + x_{k}\ot x_{j}\}_{j,k\in J}$
and consequently, by Remark \ref{rmk:hoch-M(B)}
it follows that
$ \Zep2(B,\kk)$ is linearly generated by the symmetric
functionals $\{d_{j}\ot d_{k} + d_{k}\ot d_{j}\}_{j,k\in J}$.
\end{obs}

Let $\gr A=\bigoplus_{i>0} A_i/A_{i-1}$ be the graded Hopf algebra
associated with the coradical filtration. Then $\gr A\simeq B \#\kk^{\dm}$,
where $B = (\gr A)^{\co \pi}$, and $B$ inherits the gradation
of $A$ with $\gr A(n) = B(n)\# A_{0}$.
If $x \in B(1)$, we will denote again $x=x\#1$ if no confusion arises.

Now we introduce three families of quadratic algebras given by
deformations of bosonizations of Nichols algebras. Then we prove that
these families exhaust the possible Hopf algebras with coradical $\kk^{\dm}$.
We follow the notation of \cite{AV}.

\subsubsection{Deformation of $\toba(M_{I})\# \kdm$}
For $I \in \II$, define
$\zeta_{I} = (\zeta_{i,k,p,q})_{(i,k),(p,q)\in I}$ as a family of
elements in $\kk$ such that $\zeta_{i,k,p,q}=0$ if $i\neq p$.
We call $\zeta_{I}$ a \textit{lifting datum} for $I$.

\begin{definition}\label{def:Ailambda}
Let $I\in \II$ and $\zeta_{I}$ a lifting datum for $I$.
We denote by
$\mA(\zeta_{I})$ the
$\Bbbk$-algebra given by
$(T(M_{I})\# \kk^{\dm}) / \mJ_{\zeta_{I}}$, where $\mJ_{\zeta_{I}}$ is the two-sided
ideal generated by the elements
\begin{align*}
y_{r}^{(i,k)}y_{r}^{(p,q)} + y_{r}^{(p,q)}y_{r}^{(i,k)} & ,\qquad \text{ for all }(i,k),(p,q)\in I,\ r=1,2,\\
y_{1}^{(i,k)}y_{2}^{(p,q)} + y_{2}^{(p,q)}y_{1}^{(i,k)} & -
[\zeta_{i,k,p,q}(1 -\theta_{0,q-k}) -\zeta_{i,q,p,k}\theta_{1,k-q}]
,\qquad \text{ for all }(i,k),(p,q)\in I.
\end{align*}
\end{definition}
A direct computation shows that $\mJ_{\zeta_{I}}$ is a Hopf ideal and
thus $\mA(\zeta_{I})$ is a Hopf algebra. The coalgebra structure is determined
for all $(i,k)\in I$ by
\begin{align*}
\com(y^{(i,k)}_{1}) &  =  y^{(i,k)}_{1}\ot 1 + \theta_{0, -k} \ot y^{(i,k)}_{1}
+ \theta_{1,k} \ot y^{(i,k)}_{2},\qquad \eps(y^{(i,k)}_{1}) = 0,\\
\com(y^{(i,k)}_{2}) &  =  y^{(i,k)}_{2}\ot 1 + \theta_{1,-k} \ot y^{(i,k)}_{1}
+ \theta_{0,k} \ot y^{(i,k)}_{2},\qquad \eps(y^{(i,k)}_{2}) = 0.
\end{align*}

\begin{lema}\label{lem:Ailambdadef}
$\mA(\zeta_{I})$ is a lifting of $\toba(M_{I})\# \kk^{\dm}$ and
any lifting of $\toba(M_{I})\# \kk^{\dm}$ is isomorphic to
$\mA(\zeta_{I})$ for some lifting datum $\zeta_{I}$.
In particular, $\mA(\zeta_{I})_{0} = \kdm$ and $\dim \mA(\zeta_{I}) = 4^{|I|} 2m$.
\end{lema}

\pf
We begin by proving that $\mA(\zeta_{I})$ is a lifting of $\toba(M_{I})\# \kk^{\dm}$. To see this, it is enough to show that $\mA(\zeta_{I})$
is a cocycle deformation
 of $A_{I}=\toba(M_{I})\# \kk^{\dm}$. For this, we apply the results of Section \ref{sec:cohomology}:
To produce the multiplicative cocycle, we look for an element $\mu$ in $\Hbt2(A_{I})_2$, because the relations in $\toba(M_{I})$ are quadratic.
By Theorem \ref{thm:M(B)V}, it is enough to look for $\eta \in \Zep2(\toba(M_{I}),\kk)^{\kdm}$ and take $\mu=\partial^{c}(\tilde{\eta})$.
Since by Remark \ref{rmk:H2BM},
$\Zep2(\toba(M_{I}),\kk)$ is linearly spanned by linear functionals,  by \cite[Corollary 2.6]{GM}
the exponentiation $\sigma = e^{\tilde{\eta}}$ yields a multiplicative cocycle for $A_{I}$. To obtain the deformation performed by the
multiplicative cocycle,
we need to set only the deformation in degree 2,  by Remark \ref{rmk:formula-def-part}. It turns out that the parameters describing
$ \Zep2(\toba(M_{I}),\kk)^{\kdm}$ are exactly those that produce the lifting data for $I$.

Recall that $M_{I}$ is linearly spanned by
the elements $y_{1}^{(i_{j},k_{j})}, y_{2}^{(i_{j},k_{j})}$, with $1\leq j \leq s:=|I|$,
which are homogeneous in
$\yddm$.
Let $\{d_{1}^{(i_{j},k_{j})},d_{2}^{(i_{j},k_{j})}\}_{
1\leq j\leq s}$ denote the dual basis.
By Remark \ref{rmk:H2BM}, we know that
$\Zep2(\toba(M_{I}),\kk)$ is linearly spanned by the elements
\begin{align*}
\eta_{rr}^{ikpq} & =
d_{r}^{(i,k)}\ot d_{r}^{(p,q)} + d_{r}^{(p,q)}\ot d_{r}^{(i,k)} ,
\qquad\text{ for }r=1,2,\ (i,k),(p,q) \in I,\\
\eta_{12}^{ikpq} &=
d_{1}^{(i,k)}\ot d_{2}^{(p,q)} +  d_{2}^{(p,q)}\ot d_{1}^{(i,k)},
\qquad\text{ for all }(i,k), (p,q) \in I.
\end{align*}
Let $\eta$ be a linear combination of $\eta_{rr}^{ikpq}$ and $\eta_{12}^{ikpq}$
with $(i,k), (p,q) \in I$. Note that
$\eta$ is $ \kk^{\dm}$-invariant if and only if the non-zero summands $\eta_{rr}^{ikpq}$ and $\eta_{12}^{ikpq}$
are $ \kk^{\dm}$-invariant.
Write $\eta_{rr}^{ikpq}=\eta_{rr}$.
Then
$\eta_{rr}$ is not $ \kk^{\dm}$-invariant because $$f\cdot \eta_{rr} (y_{r}^{(i,k)}, y_{r}^{(p,q)}) =
f(h^{(-1)^{r}(i+p)})\eta_{rr}(y_{r}^{(i,k)}, y_{r}^{(p,q)}),$$
and
$f(h^{(-1)^{r}(i+p)})\neq f(1)$ for all $f \in \kk^{\dm}$ since $1\leq i,p < n$, see Remark \ref{rmk:kdm-inv}.
On the other hand, $\eta_{12}^{ikpq}$ is $ \kk^{\dm}$-invariant if and only if
$f(h^{(i-p)}) = f(1)$ for all $f \in \kk^{\dm}$, that is, $i=p$.
Let $(i,k), (i,q) \in I$ and write
\begin{equation}\label{eq:eta}
\eta_{12} = \sum_{(i,k),(p,q)\in I} \frac{\zeta_{i,k,i,q}}{2}\eta_{12}^{ikiq}.
\end{equation}
Then, any possible
deformation in degree 2 is given by $\pac(\widetilde{\eta_{12}}) $.
By Remark \ref{rmk:formula-def-part}, we have
\begin{align*}
 (y_{1}^{(i,k)})^{2} & = {\eta_{12}}(y_{1}^{(i,k)},y_{1}^{(i,k)}) -
 \theta_{0,m-k}\theta_{0,m-k} {\eta_{12}}(y_{1}^{(i,k)},y_{1}^{(i,k)}) - \theta_{0,m-k}\theta_{1,k}
 {\eta_{12}}(y_{1}^{(i,k)},y_{2}^{(i,k)})\\
& \quad -\theta_{1,k}\theta_{0,m-k} {\eta_{12}}(y_{2}^{(i,k)},y_{1}^{(i,k)})- \theta_{1,k}
\theta_{1,k}{\eta_{12}}(y_{2}^{(i,k)},y_{2}^{(i,k)})
\overset{\eqref{eq:theta3}}= 0,\\
(y_{2}^{(i,k)})^{2} & = {\eta_{12}}(y_{2}^{(i,k)},y_{2}^{(i,k)}) -
 \theta_{1,m-k}\theta_{1,m-k}
{\eta_{12}}(y_{1}^{(i,k)},y_{1}^{(i,k)}) -
 \theta_{1,m-k}\theta_{0,k}
{\eta_{12}}(y_{1}^{(i,k)},y_{2}^{(i,k)})\\
& \quad -\theta_{0,k}\theta_{1,m-k}
{\eta_{12}}(y_{2}^{(i,k)},y_{1}^{(i,k)})-\theta_{0,k}\theta_{0,k}
{\eta_{12}}(y_{2}^{(i,k)},y_{2}^{(i,k)})
\overset{\eqref{eq:theta3}}= 0,
\end{align*}
since ${\eta_{12}}(y_{1}^{(i,k)},y_{1}^{(i,k)})= 0 =
{\eta_{12}}(y_{2}^{(i,k)},y_{2}^{(i,k)})$
for all $(i,k)\in I$.
Now we compute the deformation of $y_{1}^{(i,k)}y_{2}^{(p,q)}+y_{2}^{(p,q)}y_{1}^{(i,k)} $ for
$(i,k),(p,q)\in I$. For the first term we have
\begin{align*}
 y_{1}^{(i,k)}y_{2}^{(p,q)}  & = {\eta_{12}}(y_{1}^{(i,k)},y_{2}^{(p,q)}) -
(y_{1}^{(i,k)})_{(-1)}(y_{2}^{(p,q)})_{(-1)}
{\eta_{12}}((y_{1}^{(i,k)})_{(0)},(y_{2}^{(p,q)})_{(0)})
 \\
 & = {\eta_{12}}(y_{1}^{(i,k)},y_{2}^{(p,q)}) -
 \theta_{0,m-k}\theta_{1,m-q}
{\eta_{12}}(y_{1}^{(i,k)},y_{1}^{(p,q)})
 -\theta_{0,m-k}\theta_{0,q}
{\eta_{12}}(y_{1}^{(i,k)},y_{2}^{(p,q)}) \\
& \quad
 -\theta_{1,k}\theta_{1,m-q}
{\eta_{12}}(y_{2}^{(i,k)},y_{1}^{(p,q)})- \theta_{1,k}\theta_{0,q}
{\eta_{12}}(y_{2}^{(i,k)},y_{2}^{(p,q)}) \\
& =
\frac{1}{2}(\zeta_{i,k,p,q} - \zeta_{i,k,p,q}\theta_{0,m-k}\theta_{0,q} - \zeta_{i,q,p,k}\theta_{1,k}\theta_{1,m-q})\\
&= \frac{1}{2}(\zeta_{i,k,p,q} - \zeta_{i,k,p,q}\theta_{0,q-k} - \zeta_{i,q,p,k}\theta_{1,k-q}),
 \end{align*}
 since ${\eta_{12}}(y_{1}^{(i,k)},y_{1}^{(p,q)})= 0 =
 {\eta_{12}}(y_{2}^{(i,k)},y_{2}^{(p,q)})$
for all $(i,k), (p,q)\in I$.
For the second term,
\begin{align*}
y_{2}^{(p,q)}y_{1}^{(i,k)} & ={\eta_{12}}(y_{2}^{(p,q)},y_{1}^{(i,k)}) -
(y_{2}^{(p,q)})_{(-1)}(y_{1}^{(i,k)})_{(-1)}
{\eta_{12}}((y_{2}^{(p,q)})_{(0)},(y_{1}^{(i,k)})_{(0)})
 \\
 & = {\eta_{12}}(y_{2}^{(p,q)},y_{1}^{(i,k)}) -
 \theta_{1,m-q}\theta_{0,m-k}
{\eta_{12}}(y_{1}^{(p,q)},y_{1}^{(i,k)})
 -\theta_{1,m-q}\theta_{1,k}
{\eta_{12}}(y_{1}^{(p,q)},y_{2}^{(i,k)}) -\\
& \quad
 -\theta_{0,q}\theta_{0,m-k}
{\eta_{12}}(y_{2}^{(p,q)},y_{1}^{(i,k)}) -\theta_{0,q}\theta_{1,k}
{\eta_{12}}(y_{2}^{(p,q)},y_{2}^{(i,k)})\\
& =\frac{1}{2}( \zeta_{i,k,p,q} -\zeta_{i,q,p,k} \theta_{1,m-q}\theta_{1,k} -\zeta_{i,k,p,q}\theta_{0,q}\theta_{0,m-k})\\
& =\frac{1}{2}( \zeta_{i,k,p,q} - \zeta_{i,q,p,k} \theta_{1,k-q} -\zeta_{i,k,p,q}\theta_{0,q-k}).
 \end{align*}
Consequently, $\pac(\widetilde{\eta_{12}}) $ produces the deformation
$$
 y_{1}^{(i,k)}y_{2}^{(p,q)}+y_{2}^{(p,q)}y_{1}^{(i,k)}  = \zeta_{i,k,p,q}(1 -\theta_{0,q-k}) -\zeta_{i,q,p,k}\theta_{1,k-q},
$$
and $(\zeta_{i,k,p,q})_{(i,k),(p,q)\in I}$ is a lifting datum for $\mA(\zeta_I)$.
Note that any lifting datum for $\mA(\zeta_I)$ defines an element $\eta_{12}$ in $\Zep2(\toba(M_{I}),\kk)^{\kk^{\dm}}$ as in \eqref{eq:eta}.

Conversely, let $A$ be a Hopf algebra such that $\gr A =\toba(M_{I})\# \kk^{\dm} = A_{I}$.
Then by Remark \ref{rmk:lifting=def}, $A$ corresponds to a formal deformation of $A_{I}$. Write $\mu  \in \Zbh2(A_{I})$ for its infinitesimal part.
By Theorem \ref{thm:M(B)V} and the calculations above,
we have that $\mu = \partial^{c}(\tilde{\eta})$ for some $\eta \in \Zep2(\toba(M_{I}),\kk)^{\kdm}=\kk\; \eta_{12}$,
and the exponentiation $\sigma= e^{\tilde{\eta}}$ is a multiplicative cocycle
for $A_{I}$. As $\toba(M_{I})$ is a quadratic algebra, by Remark \ref{rmk:formula-def-part} the deformation corresponding
to $\eta$ coincides with the deformation given by $\sigma$. Hence, $A \simeq \mA(\zeta_I)$ for the lifting datum
associated with $\eta$.
 \epf

\begin{cor}\label{cor:mikesnotdef}
Let $A$ be a finite-dimensional Hopf algebra with $A_{0}\simeq \kk^{\dm}$ such that
its infinitesimal braiding is isomorphic to
$M_{I}$ with $I= (i,k)$ or $|I|\geq 2$ such that
$i \neq p$ for all $(i,k), (p,q) \in I$.
Then $A \simeq \toba (M_{I})\# \Bbbk^{\dm}$.
\end{cor}

\pf
By assumption, $\gr A \simeq \toba (M_{I})\#  \kk^{\dm}$
and $A$ is a lifting of $\toba (M_{I})\#  \kk^{\dm}$.
Then, by Lemma \ref{lem:Ailambdadef}, we know that
$A \simeq (T(M_{I})\#  \kk^{\dm}) / \mJ_{\zeta_{I}}$, where $\mJ_{\zeta_{I}}$ is the ideal generated by
\begin{align*}
y_{r}^{(i,k)}y_{r}^{(p,q)} + y_{r}^{(p,q)}y_{r}^{(i,k)} & ,\qquad \text{ for all }(i,k),(p,q)\in I,\ r=1,2,\\
y_{1}^{(i,k)}y_{2}^{(p,q)} + y_{2}^{(p,q)}y_{1}^{(i,k)} & -
[\zeta_{i,k,p,q}(1 -\theta_{0,q-k}) -\zeta_{i,q,p,k}\theta_{1,k-q}]
,\qquad \text{ for all }(i,k),(p,q)\in I.
\end{align*}
If $I=(i,k)$ or $i\neq p$ for all $(i,k), (p,q) \in I$, by \eqref{eq:theta1}
the second relation reduces to
$ y_{1}^{(i,k)}y_{2}^{(p,q)} + y_{2}^{(p,q)}y_{1}^{(i,k)}$.
Thus, $A \simeq \toba(M_{I})\# \kk^{\dm} $ and the corollary is proved.
\epf

\subsubsection{Deformations of $\toba(M_{L})\# \kdm$}
For $L \in \Lc$, define
$\mu_{L} = (\mu_{\ell,t})_{\ell\leq t\in L}$, $\nu_{L} = (\nu_{\ell,t})_{\ell\leq t\in L}$,
$\tau_{L} = (\tau_{\ell,t})_{\ell, t\in L}$ as families of
elements in $\kk$.
We call $(\mu_{L},\nu_{L},\tau_{L})$ a \textit{lifting data} for $L$.

\begin{definition}\label{def:BLmugamma} Let $L \in \Lc$.
Given a lifting data $(\mu_{L},\nu_{L},\tau_{L})$ for $L$
we denote by
$\mB(\mu_{L},\nu_{L},\tau_{L})$ the
algebra $(T(M_{L})\# \kk^{\dm}) / \mJ_{\mu_{L},\nu_{L},\tau_{L}}$
where $\mJ_{\mu_{L},\nu_{L},\tau_{L}}$ is the two-sided ideal generated by
\begin{align*}
x_{1}^{(\ell)}x_{1}^{(t)} + x_{1}^{(t)}x_{1}^{(\ell)}
- \mu_{\ell,t} (1 -
 \theta_{0,-\ell-t}) - \nu_{\ell,t}\theta_{1,\ell+t}
& ,\qquad \text{ for all }\ell, t \in L,\\
x_{2}^{(\ell)}x_{2}^{(t)} + x_{2}^{(t)}x_{2}^{(\ell)}
- \nu_{\ell,t} (1 -
 \theta_{0,\ell+t}) - \mu_{\ell,t}\theta_{1,-\ell-t}
 & ,\qquad \text{ for all }\ell, t \in L,\\
x_{1}^{(\ell)}x_{2}^{(t)} + x_{2}^{(t)}x_{1}^{(\ell)}  -
\tau_{\ell,t}( 1 -\theta_{0,t-\ell}) - \tau_{t,\ell} \theta_{1,\ell-t}
&,\qquad \text{ for all }\ell,t\in L.
\end{align*}
\end{definition}
Here $\pm\ell \pm t$ means $\pm\ell \pm t$  $\mod m$. Recall that $\ell, t <n$ and $2n =m$.
\medbreak

A direct computation shows that $\mJ_{\mu_{L},\nu_{L},\tau_{L}}$ is a Hopf ideal and
thus $\mB(\mu_{L},\nu_{L},\tau_{L})$ is a Hopf algebra. The coalgebra structure is determined
for all $\ell \in L$ by
\begin{align*}
\com(x^{(\ell)}_{1}) &  =  x^{(\ell)}_{1}\ot 1 + \theta_{0,-\ell} \ot x^{(\ell)}_{1}
+ \theta_{1,\ell} \ot x^{(\ell)}_{2}, \qquad
\eps(x^{(\ell)}_{1}) = 0,\\
\com(x^{(\ell)}_{2}) &  =  x^{(\ell)}_{1}\ot 1 + \theta_{1,-\ell} \ot x^{(\ell)}_{1}
+ \theta_{0,\ell} \ot x^{(\ell)}_{2},\qquad
\eps(x^{(\ell)}_{1}) = 0.
\end{align*}

\begin{lema}\label{lem:BLmugammadef}
$\mB(\mu_{L},\nu_{L},\tau_{L})$ is a lifting of $\toba(M_{L})\# \kk^{\dm}$ and
any lifting of $\toba(M_{L})\# \kk^{\dm}$ is isomorphic to $\mB(\mu_{L},\nu_{L},\tau_{L})$
 for some lifting datum $(\mu_{L},\nu_{L},\tau_{L})$.
In particular, $\mB(\mu_{L},\nu_{L},\tau_{L})_{0} = \kdm$ and $\dim \mB(\mu_{L},\nu_{L},\tau_{L}) = 4^{|L|} 2m$.
\end{lema}

\pf We proceed as in the proof of Lemma \ref{lem:Ailambdadef}. That is,
we prove
that $\mB(\mu_{L},\nu_{L},\tau_{L})$ is a lifting of $\toba(M_{L})\# \kk^{\dm}$,
by showing that it is a cocycle deformation. Similarly as before, the multiplicative $2$-cocycle is given by
$\sigma= e^{\tilde{\eta}}$ where
 $\eta \in \Zep2(\toba(M_{L}),\kk)^{\kk^{\dm}}$. Since the relations in $\toba(M_{L})$ are quadratic, we need to describe
 only the deformation in degree 2, by Remark \ref{rmk:formula-def-part}.

We know that $M_{L}$ is linearly spanned by
the elements $x_{1}^{(\ell)}, x_{2}^{(\ell)}$ with $\ell \in L$
which are homogeneous in
$\yddm$ of degree $h^{n}$.
Let $\{d_{1}^{(\ell)},d_{2}^{(\ell)}\}_{\ell \in L}$ denote the dual basis.
By Remark \ref{rmk:H2BM}, we know that
$\Zep2(\toba(M_{L}),\kk)$ is linearly spanned by the elements
\begin{align*}
\eta_{rr}^{(\ell),(t)} & =
d_{r}^{(\ell)}\ot d_{r}^{(t)} + d_{r}^{(t)}\ot d_{r}^{(\ell)} ,
\qquad\text{ for }r=1,2,\ \ell,t \in L\\
\eta_{12}^{(\ell,t)} &=
d_{1}^{(\ell)}\ot d_{2}^{(t)} +  d_{2}^{(t)}\ot d_{1}^{(\ell)},
\qquad\text{ for all }\ell,t \in L.
\end{align*}
It follows that $\eta_{rr}^{(\ell),(t)}$ and $\eta_{12}^{(\ell,t)}$
are $\kdm$-invariant for all $r=1,2$ and $\ell,t \in L$; in particular,
$\Zep2(\toba(M_{L}),\kk)^{\kdm}= \Zep2(\toba(M_{L}),\kk)$.
Define
$$ \eta = \frac{1}{2}\left(\sum_{\ell \leq t \in L}
\left(\mu_{\ell,t} \eta_{11}^{(\ell),(t)} +
\nu_{\ell,t} \eta_{22}^{(\ell),(t)}\right) +  \sum_{\ell,  t \in L}
\tau_{\ell,t}\eta_{12}^{(\ell,t)}\right).
$$
Then $\eta$ represents a generic element in $\Zep2(\toba(M_{L}),\kk)^{\kdm}$
and all the possible
$2$-deformations are given by $\pac(\widetilde{\eta}) $.
Following Remark \ref{rmk:formula-def-part}, for $\ell \leq t \in L$ we have
\begin{align*}
x_{1}^{(\ell)}x_{1}^{(t)} & = {\eta}(x_{1}^{(\ell)},x_{1}^{(t)}) -
 \theta_{0,-\ell}\theta_{0,-t}
{\eta}(x_{1}^{(\ell)},x_{1}^{(t)}) -\theta_{0,-\ell}\theta_{1,t}
{\eta}(x_{1}^{(\ell)},x_{2}^{(t)}) -\\
& \quad
 -\theta_{1,\ell}\theta_{0,-t}
 {\eta}(x_{2}^{(\ell)},x_{1}^{(t)})
  -\theta_{1,\ell}\theta_{1,t}
 {\eta}(x_{2}^{(\ell)},x_{2}^{(t)})\\
& = \frac{1}{2}(\mu_{\ell,t} -
 \theta_{0,-\ell}\theta_{0,-t}
\mu_{\ell,t} - \theta_{0,-\ell}\theta_{1,t}\tau_{\ell,t}
-\theta_{1,\ell}\theta_{0,-t}\tau_{t,\ell}
 -\theta_{1,\ell}\theta_{1,t}\nu_{\ell,t})\\
 & = \frac{1}{2}[\mu_{\ell,t} (1 -
 \theta_{0,-\ell-t}) - \theta_{1,\ell+t}\nu_{\ell,t}].
\end{align*}
Analogously,
\begin{align*}
x_{2}^{(\ell)}x_{2}^{(t)} & = {\eta}(x_{2}^{(\ell)},x_{2}^{(t)}) -
 \theta_{1,-\ell}\theta_{1,-t}
{\eta}(x_{1}^{(\ell)},x_{1}^{(t)}) -\theta_{1,-\ell}\theta_{0,t}
{\eta}(x_{1}^{(\ell)},x_{2}^{(t)}) -\\
& \quad
 -\theta_{0,\ell}\theta_{1,-t}
 {\eta}(x_{2}^{(\ell)},x_{1}^{(t)})
  -\theta_{0,\ell}\theta_{0,t}
 {\eta}(x_{2}^{(\ell)},x_{2}^{(t)})\\
& = \frac{1}{2}(\nu_{\ell,t} -
\mu_{\ell,t} \theta_{1,-\ell-t}
 - \nu_{\ell,t}\theta_{0,\ell+t})
 = \frac{1}{2}[\nu_{\ell,t} (1 -
 \theta_{0,\ell+t}) - \mu_{\ell,t}\theta_{1,-\ell-t}].
\end{align*}
Finally,
for the relation involving $x_{1}^{(\ell)}$ and $x_{2}^{(t)}$ we have that
\begin{align*}
 x_{1}^{(\ell)}x_{2}^{(t)}+x_{2}^{(t)}x_{1}^{(\ell)}  & =
{\eta}(x_{1}^{(\ell)},x_{2}^{(t)}) -
(x_{1}^{(\ell)})_{(-1)}(x_{2}^{(t)})_{(-1)}
{\eta}((x_{1}^{(\ell)})_{(0)},(x_{2}^{(t)})_{(0)}) +\\
& \quad + {\eta}(x_{2}^{(t)},x_{1}^{(\ell)}) -
(x_{2}^{(t)})_{(-1)}(x_{1}^{(\ell)})_{(-1)}
{\eta}((x_{2}^{(t)})_{(0)},(x_{1}^{(\ell)})_{(0)}).
\end{align*}
For the first term on the right hand side we have
\begin{align*}
 & {\eta}(x_{1}^{(\ell)},x_{2}^{(t)}) -
(x_{1}^{(\ell)})_{(-1)}(x_{2}^{(t)})_{(-1)}
{\eta}((x_{1}^{(\ell)})_{(0)},(x_{2}^{(t)})_{(0)})
 = \\
 & = {\eta}(x_{1}^{(\ell)},x_{2}^{(t)}) -
 \theta_{0,-\ell}\theta_{1,-t}
{\eta_{12}}(x_{1}^{(\ell)},x_{1}^{(t)})
 -\theta_{0,-\ell}\theta_{0,t}
{\eta}(x_{1}^{(\ell)},x_{2}^{(t)}) \\
& \quad
 -\theta_{1,\ell}\theta_{1,-t}
{\eta}(x_{2}^{(\ell)},x_{1}^{(t)})- \theta_{1,\ell}\theta_{0,t}
{\eta}(x_{2}^{(\ell)},x_{2}^{(t)}) \\
& =  \frac{1}{2}(\tau_{\ell,t} - \tau_{\ell,t}\theta_{0,-\ell+t} -
\tau_{t,\ell}\theta_{1,\ell-t})
=  \frac{1}{2}[\tau_{\ell,t}(1 - \theta_{0,-\ell+t}) -
\tau_{t,\ell}\theta_{1,\ell-t}],
 \end{align*}
and a similar computation for the second term yields
\begin{align*}
 & {\eta}(x_{2}^{(t)},x_{1}^{(\ell)}) -
(x_{2}^{(t)})_{(-1)}(x_{1}^{(\ell)})_{(-1)}
{\eta}((x_{2}^{(t)})_{(0)},(x_{1}^{(\ell)})_{(0)})
 = \frac{1}{2}[\tau_{\ell,t}( 1 -\theta_{0,t-\ell}) - \tau_{t,\ell} \theta_{1,\ell-t}].
 \end{align*}
Hence, the deformation of
$\toba(M_{L})\#\kdm$ performed by $\pac(\tilde{\eta})$ is isomorphic to
$\mB(\mu_{L}, \nu_{L}, \tau_{L})$ for some lifting data
$(\mu_{L}, \nu_{L}, \tau_{L})$ associated with the parameters of $\eta$.

Conversely, let $A$ be a Hopf algebra such that $\gr A =\toba(M_{L})\# \kk^{\dm}=A_{L}$.
Then by Remark \ref{rmk:lifting=def}, $A$ corresponds to a formal deformation of $A_{L}$. Write $\mu  \in \Zbh2(A_{L})$ for its infinitesimal part.
By Theorem \ref{thm:M(B)V} and the calculations above,
we have that $\mu = \partial^{c}(\tilde{\eta})$ for some $\eta \in \Zep2(\toba(M_{L}),\kk)^{\kdm}$,
and the exponentiation $\sigma= e^{\tilde{\eta}}$ is a multiplicative cocycle
for $A_{L}$. By Remark \ref{rmk:formula-def-part}, the deformation corresponding
to $\eta$ coincides with the deformation given by $\sigma$. Hence, $A\simeq \mB(\mu_{L}, \nu_{L}, \tau_{L})$ for some lifting datum
$(\mu_{L}, \nu_{L}, \tau_{L})$.
\epf

\subsubsection{Deformations of $\toba(M_{I,L})\# \kdm$}
We introduce the last families of deformed algebras corresponding to
the finite-dimensional Nichols algebras associated with mixed classes.

\begin{definition}\label{def:CIL} Let $(I,L) \in \K$.
Given lifting data $\zeta_{I}$ and $(\mu_{L},\nu_{L},\tau_{L})$ for $I$ and $L$,
respectively,
we denote by
$\mC(\zeta_{I},\mu_{L},\nu_{L},\tau_{L})$ the
algebra $(T(M_{I,L})\# \kk^{\dm}) / \mJ_{\zeta_{I},\mu_{L},\nu_{L},\tau_{L}}$
where $\mJ_{\zeta_{I}, \mu_{L},\nu_{L},\tau_{L}}$ is the two-sided ideal generated by
the elements
\begin{align*}
y_{r}^{(i,k)}y_{r}^{(p,q)} + y_{r}^{(p,q)}y_{r}^{(i,k)} & ,\qquad \forall\ (i,k),(p,q)\in I,\ r=1,2,\\
y_{1}^{(i,k)}y_{2}^{(p,q)} + y_{2}^{(p,q)}y_{1}^{(i,k)}  -
[\zeta_{i,k,p,q}(1 -\theta_{0,q-k}) -\zeta_{i,q,p,k}\theta_{1,k-q}]
&,\qquad \forall\ (i,k),(p,q)\in I,\\
x_{1}^{(\ell)}x_{1}^{(t)} + x_{1}^{(t)}x_{1}^{(\ell)}
- \mu_{\ell,t} (1 -
 \theta_{0,-\ell-t}) - \nu_{\ell,t}\theta_{1,\ell+t}
& ,\qquad \forall\ \ell, t \in L,\\
x_{2}^{(\ell)}x_{2}^{(t)} + x_{2}^{(t)}x_{2}^{(\ell)}
- \nu_{\ell,t} (1 -
 \theta_{0,\ell+t}) - \mu_{\ell,t}\theta_{1,-\ell-t}
 & ,\qquad \forall\ \ell, t \in L,\\
x_{1}^{(\ell)}x_{2}^{(t)} + x_{2}^{(t)}x_{1}^{(\ell)}  -
\tau_{\ell,t}( 1 -\theta_{0,t-\ell}) - \tau_{t,\ell} \theta_{1,\ell-t}
&,\qquad \forall\ \ell,t\in L.\\
y_{r}^{(i,k)}x_{s}^{(\ell)} + x_{s}^{(\ell)}y_{r}^{(i,k)} & ,\qquad \forall\ (i,k)\in I,\ \ell \in L,\ r,s=1,2.
\end{align*}
\end{definition}
A direct computation shows that $\mJ_{\zeta_{I},\mu_{L},\nu_{L},\tau_{L}}$ is a Hopf ideal and
thus $\mC(\zeta_{I},\mu_{L},\nu_{L},\tau_{L})$ is a Hopf algebra.
As before, for $(I,L)\in \K$, we call $(\zeta_{I},\mu_{L},\nu_{L},\tau_{L})$ a \textit{lifting
data} for $(I,L)$.

Following the same lines of the proofs of Lemmata \ref{lem:Ailambdadef} and
\ref{lem:BLmugammadef} we have

\begin{lema}\label{lem:CIL}
$\mC(\zeta_{I},\mu_{L},\nu_{L},\tau_{L})$ is a lifting of $\toba(M_{I,L})\# \kk^{\dm}$ and
any lifting of $\toba(M_{I,L})\# \kk^{\dm}$ is isomorphic to $\mC(\zeta_{I},\mu_{L},\nu_{L},\tau_{L})$
for some lifting datum $(\zeta_{I},\mu_{L},\nu_{L},\tau_{L})$.
In particular, we have $\mC(\zeta_{I},\mu_{L},\nu_{L},\tau_{L})_{0} = \kdm$
and $\dim \mC(\zeta_{I},\mu_{L},\nu_{L},\tau_{L}) = 4^{|I|+|L|} \;2m$.
\end{lema}

\pf We begin by looking for the $\kdm$-invariant cocycles in $\Zep2(\toba(M_{I,L}),\kk)$.
Since $M_{I,L} =\left(\bigoplus_{1\leq j\leq s} M_{(i_{j},k_{j})}\right)\oplus
\left(\bigoplus_{1\leq t\leq s} M_{\ell_{t}}\right)$, it is
linearly spanned by the elements
$y_{1}^{(i_{j},k_{j})}, y_{2}^{(i_{j},k_{j})}$, with $1\leq j \leq s=|I|$, and
$x_{1}^{(\ell)}, x_{2}^{(\ell)}$ with $\ell \in L$.
Let $\left\{d_{1}^{(i_{j},k_{j})},d_{2}^{(i_{j},k_{j})}\right\}_{1\leq j\leq s}
\bigcup \left\{d_{1}^{(\ell)},d_{2}^{(\ell)}\right\}_{\ell \in L}$ denote the dual basis.

By the proof of the Lemmata \ref{lem:Ailambdadef} and \ref{lem:BLmugammadef}, we know that
$\eta_{12}^{ikiq}$, $\eta_{11}^{(\ell),(t)}$, $\eta_{22}^{(\ell),(t)}$ and
$\eta_{12}^{(\ell),(t)}$ belong to $\Zep2(\toba(M_{I,L}),\kk)^{\kdm}$.
On the other hand, the elements in $\Zep2(\toba(M_{I,L}),\kk)$ given by
\begin{align*}
\eta_{rr}^{ik\ell} & = d_{r}^{(i,k)}\ot d_{r}^{(\ell)} + d_{r}^{(\ell)} \ot d_{r}^{(i,k)},
\qquad\text{ for } r=1,2,\ (i,k), \in I, \ell \in L,\\
\eta_{12}^{ik\ell} &= d_{1}^{(i,k)}\ot d_{2}^{(\ell)} +  d_{2}^{(\ell)}\ot d_{1}^{(i,k)},
\qquad\text{ for all }(i,k) \in I, \ell \in L,\\
\eta_{21}^{ik\ell} &= d_{2}^{(i,k)}\ot d_{1}^{(\ell)} +  d_{1}^{(\ell)}\ot d_{2}^{(i,k)},
\qquad\text{ for all }(i,k) \in I, \ell \in L,
\end{align*}
are not $\kdm$-invariant because
\begin{align*}
f\cdot \eta_{rr}^{ik\ell} (y_{r}^{(i,k)}, x_{r}^{(\ell)}) &= f(h^{(-1)^{r}(i+\ell)}) \,
\eta_{rr}^{ik\ell} (y_{r}^{(i,k)}, x_{r}^{(\ell)}),\\
f\cdot \eta_{12}^{ik\ell} (y_{1}^{(i,k)}, x_{2}^{(\ell)}) &= f(h^{(-1)^{r}(i+\ell)}) \,
\eta_{12}^{ik\ell} (y_{r}^{(i,k)}, x_{2}^{(\ell)}),\\
f\cdot \eta_{21}^{ik\ell} (y_{2}^{(i,k)}, x_{1}^{(\ell)}) &= f(h^{(-1)^{r}(i+\ell)}) \,
\eta_{21}^{ik\ell} (y_{r}^{(i,k)}, x_{1}^{(\ell)}),
\end{align*}
and $f(h^{(-1)^{r}(i+\ell)})\neq f(1)$ for all $f \in \kk^{\dm}$ since $1\leq i,\ell < n$, see Remark \ref{rmk:kdm-inv}.
Hence, the deformation of
$\toba(M_{I,L})\#\kdm$ performed by $\pac(\tilde{\eta}_{12})$ and $\pac(\tilde{\eta})$ is isomorphic to
$\mC(\zeta_{I},\mu_{L},\nu_{L},\tau_{L})$ for some lifting data
$(\zeta_{I},\mu_{L},\nu_{L},\tau_{L})$
associated with the parameters of $\eta_{12}$ and $\eta$.

The converse  follows \textit{mutatis mutandis} using the same arguments of the last paragraph of the proof of
Lemma \ref{lem:Ailambdadef}.
\epf

We end this paper with the proof of Theorem \ref{thm:haoverkdm}.
\bigbreak

\noindent \textit{Proof of Theorem \ref{thm:haoverkdm}.}
It is clear that two algebras from different families are not isomorphic as Hopf algebras
since their infinitesimal braidings are not isomorphic as Yetter-Drinfeld modules.

Let $A$ be a finite-dimensional Hopf algebra with coradical $A_{0}\simeq \kdm$.
Then by Theorems \ref{thm:item-c} and \ref{thm:all-nichols-dm}, we have that $\gr A \simeq \toba(M)\# \kdm$,
where $M$ isomorphic either to $M_{I}$, or to
$M_{L}$, or to $M_{I,L}$, with $I\in \II$, $L\in \Lc$ and $(I,L) \in \K$, respectively.
Hence, by Lemmata \ref{lem:Ailambdadef}, \ref{lem:BLmugammadef} and \ref{lem:CIL} the result follows.
\qed


\subsection*{Acknowledgments}
We thank N. Andruskiewitsch, I. Angiono and A. Garc\'ia Iglesias for interesting
conversations.
Research of this paper was done during mutual visits of the authors to the Math.
Department at Saint Mary's University (Canada),
the Math. Department at Universidad Nacional de La Plata, and FaMAF at Universidad
Nacional de C\'ordoba (Argentina).
We thank the people of these departments for the warm hospitality.
We also wish to thank the referee for the careful reading of the paper and the detailed suggestions that helped us to improve
the exposition.

\end{document}